\documentclass[12pt]{amsart}

\usepackage[margin = 1in]{geometry}

\usepackage{tikz}
\usepackage{tikz-cd}
\usepackage{url, hyperref}
\usepackage{float}

\tikzset{->-/.style={decoration={
  markings,
  mark=at position .45 with {\arrow{>}}},postaction={decorate}}}
  
\usetikzlibrary{shapes.arrows}
\usetikzlibrary{decorations.pathreplacing}
\usepackage[all]{xy}

\tikzset{->-/.style={decoration={
  markings,
  mark=at position .45 with {\arrow{>}}},postaction={decorate}}}

\usepackage{amsmath}
\usepackage{amssymb}
\usepackage{enumitem}
\usepackage{graphicx}
\usepackage{mathdots}
\usepackage{color}
\usepackage{diagbox}
\usepackage{array, makecell}
\usepackage{rotating}
\usepackage{amsmath}
\usepackage{tikz-cd}

\def\cD{\mathcal{D}}
\def\cV{\mathcal{V}}
\def\cW{\mathcal{W}}

\def\cO{\mathcal{O}}

\def\cM{{\mathcal{M}}}
\def\barQ{\overline{Q}}

\def\Z{\mathbb{Z}}
\def\C{\mathbb{C}}

\def\b1{{\bf 1}}

\def\E{\mathsf{E}}

\def\H{\mathsf{H}}

\def\P{\mathbb{P}}
\def\bP{\mathbb{P}}

\def\vir{{\mathsf{vir}}}
\def\cP{\mathcal{P}}

\def\uW{\underline{W}}

\def\cE{\mathcal{E}}

\def\cL{\mathcal{L}}

\def\barP{\overline{\mathcal{P}}}

\newcount\colveccount
\newcommand*\colvec[1]{
        \global\colveccount#1
        \begin{pmatrix}
        \colvecnext
}
\def\colvecnext#1{
        #1
        \global\advance\colveccount-1
        \ifnum\colveccount>0
                \\
                \expandafter\colvecnext
        \else
                \end{pmatrix}
        \fi
}

\usepackage{amsthm}
\newtheorem{definition}{Definition}[subsection]
\newtheorem{theorem}[definition]{Theorem}
\newtheorem{example}[definition]{Example}
\newtheorem{proposition}[definition]{Proposition}
\newtheorem{ingredient}[definition]{Ingredient}
\newtheorem{corollary}[definition]{Corollary}
\newtheorem{lemma}[definition]{Lemma}
\newtheorem{reduction}[definition]{Reduction}
\newtheorem{remark}[definition]{Remark}
\newtheorem{conjecture}[definition]{Conjecture}
\newtheorem{question}[definition]{Question}
\newtheorem{program}[definition]{Program}

\newtheorem{data}[definition]{Data}

\newcommand{\Pic}{\mathsf{Pic}}
\newcommand{\Sym}{\mathsf{Sym}}

\newcommand{\Gr}{\mathsf{Gr}}
\newcommand{\Fl}{\mathsf{Fl}}

\newcommand{\cI}{\mathcal{I}}
\newcommand{\cJ}{\mathcal{J}}

\newcommand{\Bl}{{\mathsf{Bl}}}
\newcommand{\Inc}{{\mathsf{Inc}}}

\newcommand{\Tev}{{\mathsf{Tev}}}

\newcommand{\Coll}{{\mathsf{Coll}}}
\newcommand{\Hom}{{\text{Hom}}}

\newcommand{\im}{\text{im}}

\newcommand{\rank}{\text{rank}}
\DeclareMathOperator{\coker}{\text{coker}}

\DeclareMathOperator{\expdim}{expdim}
\DeclareMathOperator{\Supp}{Supp}

\newcommand{\ev}{{\text{ev}}}

\newcommand{\barM}{{\overline{\mathcal{M}}}}

\usepackage{amsmath,calligra,mathrsfs}
\DeclareMathOperator{\cHom}{\mathscr{H}\text{\kern -3pt {\calligra\large om}}\,}

\newcommand{\wtQ}{\widetilde{Q}}

\def\cE{\mathcal{E}}

\def\cL{\mathcal{L}}

\def\barD{\overline{D}}
\def\barE{\overline{\mathsf{E}}}
\def\bark{\overline{k}}

\title{Complete quasimaps to $\Bl_{\P^s}(\P^r)$}

\author{Alessio Cela}
\address{ University of Cambridge, Department of pure mathematics and mathematical statistics
\hfill \newline\texttt{}
 \indent Centre for Mathematical Sciences, Wilberforce Road Cambridge, UK} \email{{\tt ac2758@cam.ac.uk}}

 \author{Carl Lian}
\address{Washington University in St. Louis, Department of Mathematics, 1 Brookings Drive
\hfill \newline\texttt{}
 \indent  St. Louis, MO 63130} \email{{\tt clian@wustl.edu}}

\date{\today}

\usepackage{graphicx}

\begin{document}

\maketitle

\begin{abstract}
We introduce a moduli space of ``complete quasimaps'' to $X=\Bl_{\P^s}(\P^r)$. The construction, following previous work for curves on projective spaces, essentially proceeds by blowing up Ciocan-Fontanine--Kim's space of quasimaps at loci where sections of line bundles are linearly dependent. We conjecture that tautological intersection numbers on these moduli spaces give enumerative counts of curves of fixed complex structure on $X$ subject to general incidence conditions, in contrast with traditional compactifications of the moduli spaces of maps. A result of Farkas guarantees that these spaces are pure of expected dimension. The conjecture is proven in dimension 2, where the main input is a Brill-Noether theorem for general curves on toric surfaces.
\end{abstract}

\setcounter{tocdepth}{1}

\tableofcontents

\section{Introduction}

\subsection{Counting curves}\label{sec:intro_count}

Let $X$ be a geometric space, such as a complex projective variety. How many curves $C$ lie in $X$, perhaps subject to incidence conditions of the form $C\ni x$, where $x\in X$ is a fixed point? As a concrete example, consider the following question, with which we will be concerned in this paper.

\begin{question}\label{q:question_main_intro}
Fix a smooth, projective curve $C$ of genus \( g \), and let \( p_1, \dots, p_n \in C \) be general points. Let \( q, x_1, \dots, x_n \in \P^2 \) be general points in the plane. How many degree \( d \) plane curves \( f \colon C \to \P^2 \) are there, with multiplicity \( k \) at \( q \), and in addition with \( f(p_i) = x_i \) for all \( i = 1, \dots, n \)?
\end{question}

A plane curve \( f \colon C \to \P^2 \) of degree \( d \) and multiplicity \( k \) at \( q \) is equivalently a map to the blow-up \( f \colon C \to X=\Bl_q \P^2 \) in the homology class \( \beta=d\H - k\E \) of curves meeting a general line in $d$ points and the exceptional divisor in $k$ points. For example, in the special case of rational plane curves $(g=0)$, under a necessary numerical constraint \eqref{dim_constraint}, the answer to Question \ref{q:question_main_intro} is known to be $\binom{n-d-1}{k}$ \cite[Theorem 1.12]{cl2}.

Traditionally, such problems are approached via intersection theory on moduli spaces. Let \( \cM \) be the moduli space parameterizing maps \( f \colon C \to X \), and let \( \cM_i \subset \cM \) be the subvariety parameterizing those maps for which \( f(p_i) = x_i \). Then, the desired count of curves is
\[
\#\left(\bigcap_{i=1}^n \cM_i \subset \cM\right).
\]

When $\bigcap_{i=1}^n \cM_i$ is expected to be finite, intersection theory allows one to convert this set-theoretic intersection number into a cohomological calculation on \( \cM \). However, this is possible only after replacing \( \cM \) with a compactification \( \barM \supset \cM \), such as the moduli space of stable maps in Gromov-Witten theory, and the subvarieties $\cM_i\subset \cM$ with suitable enlargements $\barM_i\subset \barM$. Indeed, integration of cohomology classes is only sensible on compact spaces, yet \( \cM \) is not compact: maps may acquire singularities or base-points upon degeneration. A choice of compactification $\barM$ amounts to a choice of objects into which maps $f:C\to X$ are allowed to degenerate.

From here, one hopes that 
\[
\#(\text{maps})=\#\left(\bigcap_{i=1}^n \cM_i \subset \cM\right)=\#\left(\bigcap_{i=1}^n \barM_i \subset \barM\right)=\int_{\barM} \prod_{i=1}^n[\barM_i],
\]
where $[\barM_i]\in H^{*}(\barM)$ denote the cycle classes.

However, for many well-studied compact moduli spaces, such as the moduli spaces of stable maps and stable quasimaps, it is often the case that the new intersection in $\barM$ is larger than desired, including possibly infinitely many degenerate maps beyond the finite number of ``honest maps'' we wish to enumerate. While the integral on the right hand side can be made sensible by integrating against a virtual fundamental class, its value can no longer be expected to answer Question \ref{q:question_main_intro}. That is, the cohomological calculation produces only a \emph{virtual count}, which enumerates not only maps $f:C\to X$ out of a smooth curve, but rather includes excess contributions from a possibly infinite set of maps $f:C\to X$ with singular domain. Isolating only the former maps is difficult.

This paper is concerned with confronting this problem. Namely, we aim to construct compact moduli spaces of maps to $X$ on which integrals of products of cycle classes corresponding to incidence conditions are \emph{enumerative}, that is, give an actual count of smooth curves on $X$. In the next section, we state the precise motivating problem.

\subsection{Tevelev degrees}

We are interested in the following invariants.

\begin{definition}\label{def:tev}
Let $X$ be a smooth, projective variety. Let $\cM_{g,n}(X,\beta)$ be the moduli space of maps $f:C\to X$ from a smooth, pointed curve $(C,p_1,\ldots,p_n)$ such that $f_{*}[C]=\beta\in H_2(X,\Z)$.\footnote{We assume that $2g-2+n>0$ so that $\cM_{g,n}$ is Deligne-Mumford.} Let
\begin{equation*}
    \tau:\cM_{g,n}(X,\beta)\to \cM_{g,n}\times X^n
\end{equation*}
be the map defined by $\tau(f)=([(C,p_1,\ldots,p_n)],f(p_1),\ldots,f(p_n))$.

Assume\footnote{The numerical hypothesis \eqref{dim_constraint} amounts to the assumption that $\tau$ has expected relative dimension zero.}
\begin{equation}\label{dim_constraint}
    \beta\cdot K_X^\vee=\dim(X)\cdot (n+g-1),
\end{equation}
and furthermore that every component of $\cM_{g,n}(X,\beta)$ dominating the target under $\tau$ is generically smooth of expected dimension.\footnote{The transversality hypothesis is guaranteed when $n\ge g+1$ by \cite[Proposition 13]{lian_pand}.} Then, the \emph{geometric Tevelev degree} $\Tev^X_{g,n,\beta}$ of $X$ is by definition the degree of $\tau$.
\end{definition}

In other words, given a fixed, general curve $(C,p_1,\ldots,p_n)$ and general points $x_1,\ldots,x_n\in X$, the geometric Tevelev degree $\Tev^X_{g,n,\beta}$ counts the number of maps $f:C\to X$ in class $\beta$ with $f(p_i)=x_i$ for all $i$. For example, Question \ref{q:question_main_intro} asks for the value of $\Tev^{\Bl_q(\bP^2)}_{g,n,d\H-k\E}$. One can more generally impose incidence conditions with respect to higher-dimensional subvarieties of $X$. 

\begin{remark}\label{rem:history}
    The anachronistic term ``Tevelev degree'' was introduced in \cite{cps} after work of Tevelev \cite{tev}, but in fact, \emph{virtual} formulations of the question go back to the study of Vafa-Intriligator formulas in the mathematical literature, see for example \cite{bdw,st,mo} or \cite{bp,Cela2023} for more recent work.
\end{remark}

As discussed in \S\ref{sec:intro_count}, the moduli space of maps $\cM_{g,n}(X,\beta)$ is essentially never proper, so cannot support an intersection-theoretic calculation that computes $\Tev^X_{g,n,\beta}$. The principal difficulty, then, is that standard compactifications of $\cM_{g,n}(X,\beta)$, such as the space $\barM_{g,n}(X,\beta)$ of stable maps (or, as will be more pertinent here, the spaces of quasimaps), may introduce extra points in the generic fiber of $\tau$, so the enumerative counts $\Tev^X_{g,n,\beta}$ are not immediately accessible.\footnote{See \cite{lian_pand,bllrst} for detailed investigations into this phenomenon, and in particular \cite[Example 1.3]{bllrst} for the simplest example of the failure of enumerativity.} As a result, many existing calculations of $\Tev^X_{g,n,\beta}$ are restricted to large degree, e.g. \cite{lian_hyp,cl2}.

Therefore, alternate compactifications of moduli spaces of maps are needed to compute $\Tev^X_{g,n,\beta}$ and related invariants. In this paper, we construct a candidate compactification $\wtQ_{\beta}(C,X)$ of the moduli of maps to the blow-up $X=\Bl_{\P^s}(\P^r)$ of $\P^r$ along a linear subspace, which we call the moduli space of \emph{complete quasimaps}.

\begin{conjecture}\label{conj:main_intro}
Let $X=\Bl_{\P^s}(\P^r)$, and let $C$ be a general curve of genus $g$. Let $\wtQ_\beta(C,X)$ be the moduli space of complete quasimaps $u:C\to X$. If $n\ge r+1$, then
\begin{equation*}
    \Tev^X_{g,n,\beta}=\int_{\wtQ_\beta(C,X)}[\Inc(p,x)]^n,
\end{equation*}
where $\Inc(p,x)\subset \wtQ_\beta(C,X)$ is the closure of the locus of non-degenerate maps $f$ with $f(p)=x$.
\end{conjecture}

A non-degenerate map is one whose image does not lie in any hyperplane, see Definition \ref{def:nondegen}. The mild hypothesis $n\ge r+1$ ensures that the maps being counted are automatically non-degenerate. Thus, the conjecture predicts that geometric Tevelev degrees of $X=\Bl_{\P^s}(\P^r)$ are given by tautological intersection numbers on $\wtQ_\beta(C,X)$. See Conjecture \ref{conj:main} for a more general statement. We prove the conjecture in the first interesting case.

\begin{theorem}\label{thm:main_intro}
Conjecture \ref{conj:main_intro} holds in dimension 2, for $X=\Bl_q(\P^2)$.
\end{theorem}

In particular, our new moduli space $\wtQ_\beta(C,\Bl_q(\P^2))$ ``knows'' the answer to Question \ref{q:question_main_intro} in an explicit way.

Conjecture \ref{conj:main_intro} and Theorem \ref{thm:main_intro} should be viewed as transversality statements. Namely, the closed subvarieties of $\wtQ_\beta(C,X)$ corresponding to the constraints $f(p_i)=x_i$ (playing the roles of $\barM_i\subset \barM$ from \S \ref{sec:intro_count}) intersect transversely, and only in the open locus of honest maps $f:C\to X$. Our complete quasimaps construction builds on a previous construction for $X=\bP^r$ \cite{lian_pr}, and suggests a path forward for many other targets $X$. We discuss this in the next section.

\begin{remark}
    The surface $X=\Bl_q(\P^2)$ is the only Hirzebruch surface for which the invariants $\Tev^X_{g,n,\beta}$ are not fully known \cite{cl_hirzebruch}, see also \cite{cil} for related calculations. 
\end{remark}

\subsection{Brill-Noether theorems and complete quasimaps}

In the language of \S\ref{sec:intro_count}, the basic approach to removing unwanted components in intersections
\[
\bigcap_{i=1}^n \barM_i \subset \barM
\]
is to blow up the moduli space $\barM$ along the closed subvarieties where the excess intersections occur, and replace the $\barM_i$ with their proper transforms. To address Question \ref{q:question_main_intro} and its variants, we propose the following program, taking the initial compactification $\barM$ to be the moduli space of quasimaps.

\begin{program}\label{program}
\,
\begin{enumerate}
\item Fix a general curve $C$, and begin with the moduli space of quasimaps \cite{cfk} $\barM=\barQ_\beta(C,X)$. \item Prove a \underline{Brill-Noether theorem} for (quasi)maps $f:C\to X$. That is, identify conditions under which $f:C\to X$ necessarily deforms in a family of expected dimension.
\item Blow up the loci on $\barQ_\beta(C,X)$ where the conditions of the previous step are \emph{not} satisfied, and prove that tautological intersection numbers of the blown up moduli space are enumerative.
\end{enumerate}
\end{program}

The role of the Brill-Noether theorem in step (2) is to identify ``bad'' loci of $\barM$, where the subvarieties $\barM_i\subset \barM$ are most likely to have excess intersection. The classical Brill-Noether theorem \cite{GH80} concerns the case $X=\bP^r$, and asserts that a map $f:C\to \bP^r$ deforms in a family of expected dimension unless it is \emph{degenerate}, that is, its image is contained in a hyperplane in $\bP^r$. It is of independent interest to prove statements of this form for other targets.

Using the classical Brill-Noether theorem, Program \ref{program} has been carried out to compute the Tevelev degrees $\Tev^{X}_{g,n,\beta}$ for $X=\bP^r$ in \cite{lian_pr}. Namely, the iterated blow-up $\wtQ_d(C,\P^r)\to \barQ_\beta(C,\bP^r)$ of the moduli space of quasimaps along the loci where the Brill-Noether theorem fails (degenerate quasimaps) resolves all excess intersections in the calculation of $\Tev^{X}_{g,n,\beta}$. The resulting moduli space $\wtQ_d(C,\P^r)$ is fibered over the moduli space $G^r_d(C)$ of linear series on $C$, and the fibers are moduli spaces of \emph{complete collineations} \cite{vainsencher,thaddeus}.

The starting point of this paper is an analog of the Brill-Noether theorem for maps $f:C\to X=\Bl_{\P^s}(\P^r)$, which follows from a result of Farkas \cite[Theorem 0.1]{farkas}. Namely, $f$ moves in a family of expected dimension as long as its image does not lie in any hyperplane upon blow-down to $\P^r$, see Corollary \ref{cor:maps_expdim}. Therefore, the space of quasimaps $\barQ_\beta(C,X)$ is blown up along the loci where the sections defining $f$ are linearly dependent, see \S\ref{sec:complete} for the precise definition. A new phenomenon in the case $X=\Bl_{\P^s}(\P^r)$ is the need to keep track of a subspace of sections corresponding to a secant of $C$ in the course of the blow-ups. Our space of complete quasimaps $\wtQ_\beta(C,X)$ is the resulting iterated blow-up. 

\begin{remark}
Strictly speaking, we depart slightly from Program \ref{program} as stated. Rather, we start with a moduli space of \emph{augmented quasimaps} $Q^+(C,X)$ refining the moduli space of quasimaps, which is already known to be of expected dimension. Then, we blow up the space $Q^+(C,X)$. We expect that our moduli spaces can be equivalently obtained by blowing up the ordinary space of quasimaps, having the effect of blowing up components of larger than expected dimension ``out of existence,'' but do not give a proof. We will ignore this distinction in the remainder of the introduction.
\end{remark}

\begin{remark}
A key feature of our construction of $\wtQ_\beta(C,X)$ is that it is uniform: the blow-up loci make no reference to the points of $C,X$ at which incidence conditions are imposed. (In the language of \S\ref{sec:intro_count}, the blow-up of $\barM$ makes no reference to the subvarieties $\barM_i$ arising in the intersection problem.) While it is true for abstract reasons that excess intersections can always be resolved by sufficiently many blow-ups, the content of Conjecture \ref{conj:main_intro} is that there is a single sequence of blow-ups that suffices for all enumerative problems in question.
\end{remark}

As we will see, the main obstruction to proving Conjecture \ref{conj:main_intro} in full generality is that, in order to ensure that unwanted intersections in the boundary are fully resolved, we need slightly stronger Brill-Noether statements for maps $f:C\to X$, roughly that additional vanishing conditions on the underlying sections impose the expected number of conditions. When $X=\P^r$, the \emph{pointed} Brill-Noether theorem suffices. When $X=\Bl_q(\P^2)$, we invoke a Brill-Noether theorem for maps to further (toric) blow-ups of $X$ developed in our work \cite{cl_bn_surface}, see \S\ref{sec:surface_BN}-\ref{sec:toric_surface_BN}. In higher dimensions, a suitable analog is missing, but we have raised the question of establishing Brill-Noether theorems for maps to toric varieties in \cite[Question 1.5]{cl_bn_surface}.

We defer calculations of integrals on our spaces $\wtQ_\beta(C,X)$ to later work. As we will see, our construction in dimension 2 is sufficiently concrete, involving degeneracy loci of expected dimension inside well-understood moduli spaces, that a direct approach is plausible. In higher dimension, degeneration techniques along the lines of \cite{lian_pr} may be more robust.

We expect constructions of compactified moduli spaces of maps and enumerative calculations in this vein to be possible for other target varieties. The moduli space of quasimaps $\overline{Q}_\beta(C,X)$ is defined for GIT quotients $X$ of an affine variety by a reductive group \cite{cfkm}, including Grassmannians \cite{mop}, toric varieties \cite{cfk}, and hypersurfaces, so Program \ref{program} is sensible in all of these cases.

The roadmap of the paper is as follows. The moduli spaces in question are constructed in \S\ref{sec:moduli_spaces}. Theorem \ref{thm:main_intro} is proven in \S\ref{sec:set-theoretic}-\ref{sec:dim2_proof}. We discuss how to extend the method of proof (and what is missing) to higher dimension in \S\ref{sec:higher_dim}. Finally, we discuss in \S\ref{sec:other_spaces} the relationship between our new space $\wtQ_\beta(C,X)$ and other compactifications of moduli of maps $C\to X$, including the moduli spaces used in our previous work on \emph{large degree} counts of curves on $X=\Bl_q(\P^r)$ \cite{cl2}.

\subsection{Conventions}

\begin{itemize}
\item We work over $\C$.
\item The projective bundle $\P(\cE)$ is the moduli space of \emph{lines} in the fibers of $\cE$.
\item Angle brackets $\langle - \rangle$ denote linear span.
\item The support of an effective divisor $D\subset C$ on a curve is denoted $\Supp(D)$. 
\item If $X$ is a projective variety and $D$ is an effective Cartier divisor, a ``defining section'' $s\in H^0(X,\cO(D))$ is one with $\text{div}(s)=D$.
\end{itemize}

\subsection{Acknowledgments}
Portions of this work were carried out during the authors' visits at ETH Z\"{u}rich, University of Illinois at Urbana-Champaign, Tufts University, and University of Cambridge. We are grateful to these institutions for their hospitality. We thank Gavril Farkas, Alina Marian, Davesh Maulik, George McNinch, Rahul Pandharipande, Dhruv Ranganathan, Montserrat Teixidor i Bigas, and Richard Thomas for conversations related to this work. C.L. has been supported by NSF Postdoctoral Fellowship DMS-2001976 and an AMS-Simons travel grant. A.C. is supported by SNF-P500PT-222363.

\section{Moduli spaces}\label{sec:moduli_spaces}

\subsection{Maps}\label{sec:maps}

Fix, once and for all, a general curve $C$ of genus $g$, by which we mean a general point $[C]\in \cM_g$. In particular, $C$ is smooth, projective, and connected. While our constructions will be valid in families of curves, our transversality statements will only hold for fixed, general curves. We thus stick to the setting of a fixed, general curve throughout this paper.

Let $X=\Bl_{\P^s}(\P^r)$ be the blow-up of $\P^r$ along a linear subspace $P\cong\P^s$ of dimension $s\le r-2$.  Let $g:X\to\P^{r-s-1}$ be the projection from $P$, let $b:X\to\P^r$ be the blow-up, and let $E\subset X$ be the exceptional divisor. Without loss of generality, we may assume that $P$ is torus-invariant, so that $X$ is toric. Write $\H,\E\in H^2(X)$ for the divisor classes of the pullback of the hyperplane from $\P^r$ and the exceptional divisor, respectively. 

Let $\beta\in H_2(X,\Z)$ be a curve class. From the point of view of enumerative problems, the $\beta$ of interest are classes of irreducible, effective curves contained generically in the (toric) interior of $X$. We thus consider only non-zero classes $\beta=d\H^\vee+k\E^\vee$, where $d\ge k\ge 0$. The image under $b$ of a curve in $X$ in class $\beta$ has degree $d$, and intersects $P$ with multiplicity $k$. We fix such a class $\beta$ throughout this section.

\begin{definition}
    Let $\cM_\beta(C,X)$ be the moduli space of maps $f:C\to X$ with $f_{*}[C]=\beta$. We say that such an $f$ is ``in class $\beta$.'' The \emph{expected dimension} of $\cM_\beta(C,X)$ is
    \begin{equation*}
        \chi(C,f^{*}T_X)=\beta\cdot K_X^\vee - r(g-1)= (r+1)d-(r-s-1)k-r(g-1).
    \end{equation*}
\end{definition}

The expected dimension is a lower bound for the actual dimension of every component. A map $f\in\cM_\beta(C,X)$ is given by the following data. (See  \cite{cox} for a description of maps to any toric variety.)
\begin{data}\label{data:map_to_X}
\quad
\begin{itemize}
    \item line bundles $\cL_d\in \Pic^d(C)$ and $\cL_k\in \Pic^k(C)$,
    \item a section $u_0\in H^0(C,\cL_k)$
    \item sections $u_1,\ldots,u_{r-s}\in H^0(C,\cL_d\otimes \cL_k^{-1})$, and
    \item sections $u_{r-s+1},\ldots,u_{r+1}\in H^0(C,\cL_d)$.
\end{itemize}
Given fixed $\cL_d,\cL_k$, the sections $u_j$ are taken up to the $(\mathbb{C}^{*})^2$ action given by
\begin{equation*}
    (\lambda_1,\lambda_2)\cdot (u_0,u_1,\ldots,u_{r-s},u_{r-s+1},u_{r+1})=(\lambda_1u_0,\lambda_2u_1,\ldots,\lambda_2u_{r-s},\lambda_1\lambda_2u_{r-s+1},\ldots,\lambda_1\lambda_2u_{r+1}).
\end{equation*}
\end{data}

Data \ref{data:map_to_X} give rise to a map $f$ if and only if the following ``base-point-free'' conditions hold:
\begin{enumerate}
    \item[(BPF1)] $u_1,\ldots,u_{r-s}\in H^0(C,\cL_d\otimes \cL_k^{-1})$ have no common vanishing point on $C$,
    \item[(BPF2)] $u_0\in H^0(C,\cL_k)$ and $u_{r-s+1},\ldots,u_{r+1}\in H^0(C,\cL_d)$ have no common vanishing point on $C$.
\end{enumerate}

 Under conditions (BPF1) and (BPF2), the composition $g\circ f:C\to\P^{r-s-1}$ is given by the tuple $[u_1:\cdots:u_{r-s}]$, and the composition $b\circ f:C\to\P^r$ is given by the tuple
\begin{equation*}
    [u_0u_1:\ldots:u_0u_{r-s}:u_{r-s+1}:\ldots:u_{r+1}].
\end{equation*}
The sections $u_0,\ldots,u_{r+1}$ are pullbacks under $f:C\to X$ of particular torus-invariant sections on $X$. The section $u_0$ is the pullback of a section $1_E$ of $\cO_X(\E)$, the sections $u_1,\ldots,u_{r-s}$ are pullbacks of sections $y_1,\ldots,y_{r-s}$ of $\cO_X(\H-\E)$, and $u_{r-s+1},\ldots,u_{r+1}$ are pullbacks of sections $y_{r-s+1},\ldots,y_{r+1}$ of $\cO_X(\H)$. The sections $1_E,y_j$ are fixed throughout this paper.

The following example shows that $\cM_\beta(C,X)$ may have larger than expected dimension.

\begin{example}\label{eg:bad_components}
 The locus on $\cM_\beta(C,\Bl_{\P^1}(\P^3))$ where $u_{4}=0$ is isomorphic to $\cM_\beta(C,\Bl_{q}(\bP^{2}))$, which (if non-empty) has dimension \emph{at least} $3d-k-2g+2$. This exceeds $\expdim \cM_\beta(C,X)=4d-k-3g+3$ if $d<g-1$.
\end{example}
The non-emptiness of $\cM_\beta(C,\Bl_{q}(\bP^{2}))$ can be arranged by \cite[Theorem 0.5]{farkas}.

\subsection{Quasimaps}\label{sec:quasimaps}

In order to do intersection theory, we consider compactifications of $\cM_\beta(C,X)$. The starting point is the space of ($\epsilon$-stable) quasimaps of Ciocan-Fontanine--Kim \cite{cfk}.

\begin{definition}\label{def:quasimaps}
A \emph{quasimap} $u:C\to X$ \footnote{It is slightly abusive to write $u:C\to X$, as $u$ is not an actual map. We will reserve the letter $f$ for honest morphisms $f:C\to X$, and allow ourselves the abuse of notation in dealing with quasimaps.} in class $\beta$ is given by Data \ref{data:map_to_X} satisfying:
\begin{enumerate}
    \item[(NZ1)] $u_1,\ldots,u_{r-s}\in H^0(C,\cL_d\otimes \cL_k^{-1})$ are not all zero,
    \item[(NZ2)] $u_0\in H^0(C,\cL_k)$ and $u_{r-s+1},\ldots,u_{r+1}\in H^0(C,\cL_d)$ are not all zero.
\end{enumerate}
Let $\barQ_\beta(C,X)$ be the moduli space of quasimaps in class $\beta$. We say that $u\in \barQ_\beta(C,X)$ is \emph{bpf} if $u$ lies in the open subset $\cM_\beta(C,X)\subset \barQ_\beta(C,X)$ where conditions (BPF1), (BPF2) hold.
\end{definition}

$\barQ_\beta(C,X)$ is projective: it admits a natural morphism to $\Pic^d(C)\times\Pic^k(C)$ with projective fibers remembering $(\cL_d,\cL_k)$. The product $C\times \barQ_\beta(C,X)$ carries tautological line bundles $\barP_d,\barP_k$ restricting at any quasimap $u$ to the underlying line bundles $\cL_d,\cL_k$, respectively. (We reserve the term ``Poincar\'{e} line bundle'' later for a tautological bundle on $C\times\Pic^d(C)$ which is in addition trivialized along a point of $C$.) The product $C\times \barQ_\beta(C,X)$ carries universal sections
\begin{align*}
    &\overline{u}_0: \cO \to \barP_k,\\
    &\overline{u}_i: \cO \to \barP_d \otimes \barP_{k}^{-1}, \quad i=1,\ldots,r-s,\\
    &\overline{u}_i:\cO \to \barP_d, \quad \quad \quad \quad i=r-s+1,\ldots,r+1.
\end{align*}

A quasimap $u$ is a ``map with base-points.'' Indeed, sections $u_j$ satisfying (NZ1), (NZ2) but not both of (BPF1) and (BPF2) give rise to a \emph{rational} map $f:C\dashrightarrow X$. The rational map $f$ extends canonically, after twisting down base-points, to a morphism $f':C\to X$, but $f'$ no longer lies in class $\beta$ unless $u\in \cM_\beta(C,X)$ to begin with. A version of this twisting is made explicit in \S\ref{sec:twisting}.

\subsection{Non-degenerate maps and inclusions of linear series}

By Example \ref{eg:bad_components}, $\cM_\beta(C,X)$ (and thus $\barQ_\beta(C,X)$) may have components of larger than expected dimension. We wish to remove these components in order to compute enumerative (as opposed to virtual) invariants such as $\Tev^X_{g,n,\beta}$. A result of Farkas \cite[Theorem 0.1]{farkas} identifies the ``good'' maps.

\begin{definition}\label{def:nondegen}
   Let $f:C\to X$ be a morphism. We say that $f$ is \emph{non-degenerate} if the image $b(f(C))\subset \P^r$ is not contained in any hyperplane. Let $\cM^\circ_\beta(C,X)\subset \cM_\beta(C,X)$ be the open subset parametrizing non-degenerate maps.
\end{definition}

In terms of Data \ref{data:map_to_X}, a map $f$ is non-degenerate if and only if the sections
\begin{equation*}
    u_0u_1,\ldots,u_0u_{r-s},u_{r-s+1},\ldots,u_{r+1}\in H^0(C,\cL_d)
\end{equation*}
are linearly independent. In particular, we have $u_0\neq 0$. A non-degenerate map $f$ therefore gives rise to the following additional data.

\begin{definition}\cite[Definition 4.2]{chz}\label{def:G}
    Let $G^{r,r-s-1}_{d,k}(C)$ be the space parametrizing:
    \begin{itemize}
        \item a line bundle $\cL\in\Pic^d(C)$ and an effective divisor $D\in\Sym^k(C)$,
        \item an $(r+1)$-dimensional subspace (linear series) $W\subset H^0(C,\cL)$, and
        \item an $(r-s)$-dimensional subspace $V\subset W$, consisting of sections that vanish along $D$.
    \end{itemize}
\end{definition}

Indeed, given $f\in \cM^\circ_\beta(C,X)$, the data of $\cL=\cL_d$, $D=\text{div}(u_0)$, and
\begin{equation*}
(V\subset W)=(\langle u_0u_1,\ldots,u_0u_{r-s}\rangle \subset \langle u_0u_1,\ldots,u_0u_{r-s},u_{r-s+1},\ldots,u_{r+1}\rangle)
\end{equation*}
give a point of $G^{r,r-s-1}_{d,k}(C)$. Let $\pi:\cM^\circ_\beta(C,X)\to G^{r,r-s-1}_{d,k}(C)$ denote this association. A point $(\cL,D,V\subset W)\in G^{r,r-s-1}_{d,k}(C)$ is referred to as an \emph{inclusion of linear series} (ILS). We often write $\uW=(\cL,D,V\subset W)$ for brevity. Note also that Definition \ref{def:G} makes sense for any $s$ in the range $-1\le s\le r$, even though we require $0\le s\le r-2$ in order to speak of $X=\Bl_{\P^s}(\P^r)$.

There is a forgetful map $G^{r,r-s-1}_{d,k}(C)\to \Pic^d(C)$ remembering $\cL$. Choose a point $p\in C$. There is a \emph{unique} (up to isomorphism) Poincar\'{e} line bundle $\cP_d$ on $C\times\Pic^d(C)$ with the property that the restriction of $\cP_d$ to $\{p\}\times\Pic^d(C)$ is trivial. We denote the pullback of $\cP_d$ to $C\times G^{r,r-s-1}_{d,k}(C)$ also by $\cP_d$. We contrast the notation $\cP_d$ with the notation $\barP_d$ used for the tautological line bundle on $C\times \barQ_\beta(C,X)$ used in \S\ref{sec:quasimaps}; $\barP_d$ is \emph{not} assumed trivialized upon restriction to a chosen point.

We also have a forgetful map $G^{r,r-s-1}_{d,k}(C)\to \Sym^k(C)$ remembering $D$, and a universal divisor $\cD\subset C\times \Sym^k(C)$. A choice of defining section of $\cO(\cD)$ gives, after pullback and restriction, a \emph{consistent} choice of defining section $1_D\in H^0(C,\cO(D))$ for every point $\uW\in G^{r,r-s-1}_{d,k}(C)$. We fix such a choice throughout.

It is often useful to view $V$ either as a subspace of $H^0(C,\cL(-D))$ or as a subspace of $H^0(C,\cL)$. One is embedded in the other via the section $1_D$, so there is essentially no distinction, and we will pass freely between these perspectives when there is no harm. When we consider questions of vanishing of elements $v\in V$ at points $p\in C$, we will be more careful to specify the line bundle in which we view $v$ as a section.

$G^{r,r-s-1}_{d,k}(C)$ is constructed in \cite[\S 4]{chz} as a degeneracy locus, and in particular is projective of expected dimension
\begin{align*}
 &\hspace{0.5em}\dim(G^r_d(C))+\dim(\Sym^k(C))+\dim(\Gr(r-s,r+1))-k(r-s)\\
 =&\hspace{0.5em} g-(r+1)(g-d+r)+k+(r-s)(s+1)-k(r-s)\\
 =&\hspace{0.5em} \expdim \cM_\beta(C,X)-[(r-s)^2+(s+1)(r+1)-1].
\end{align*}

\begin{proposition}\label{prop:G_expdim}
    $G^{r,r-s-1}_{d,k}(C)$ is pure of expected dimension.
\end{proposition}

\begin{proof}
Consider the forgetful map
\begin{equation*}
    p:G^{r,r-s-1}_{d,k}(C)\to G^r_d(C)\times\Sym^k(C).
\end{equation*}
The fiber $p^{-1}((W,D))$ over any point is the Grassmannian of $(r-s)$-planes in
\begin{equation*}
K(W,D):=\ker(W\subset H^0(C,\cL)\to H^0(C,\cL|_D)).
\end{equation*}
We have $\dim K(W,D)\ge (r+1)-k$. When $0\le \ell\le \min(k,r+1)$, Farkas \cite[\S 2]{farkas} proves that the locus of $(W,D)$ where $\dim K(W,D) \ge (r+1)-\ell$, is pure of dimension equal to $\expdim G^{r,r-\ell}_{d,k}(C)$, and is empty when this quantity is negative. The quantity $\expdim G^{r,r-\ell}_{d,k}(C)$ is strictly increasing in $\ell$ in the range $0\le \ell \le \min(k,r+1)$, so in fact is the dimension of the locus where $\dim K(W,D)=(r+1)-\ell$.

Therefore, the dimension of the stratum of $G^{r,r-s-1}_{d,k}(C)$ mapping under $p$ to the locus where $\dim K(W,D)$ has dimension exactly $(r+1)-\ell$ equals
\begin{equation*}
\dim \Gr(r-s,r+1-\ell)+\expdim G^{r,r-\ell}_{d,k}(C),
\end{equation*}
and is empty when either summand is negative. This quantity attains a unique maximum at  $\ell=\min(k,s+1)$ (the largest possible), in which case it equals $\expdim G^{r,r-s-1}_{d,k}(C)$.
\end{proof}

\begin{corollary}\label{cor:maps_expdim}
    $\cM^\circ_\beta(C,X)$, if non-empty, is pure of expected dimension.
\end{corollary}

\begin{proof}
For any $f\in \cM^\circ_\beta(C,X)$, write $\pi(f)=\uW=(\cL,D,V\subset W)$. Then, the fiber $\pi^{-1}(\uW)$ parametrizes Data \ref{data:map_to_X} satisfying (BPF1) and (BPF2), where:
\begin{itemize}
    \item $u_0\in H^0(C,\cO(D))$ is a defining section,
    \item $u_0u_1,\ldots,u_0u_{r-s}$ form a basis of $V$, and
    \item $u_0u_1,\ldots,u_0u_{r-s},u_{r-s+1},\ldots,u_{r+1}$ form a basis of $W$. 
\end{itemize}
It follows that the dimension of $\pi^{-1}(\uW)$ equals
\begin{equation*}
    1+(r-s)^2+(s+1)(r+1)-2=\expdim \cM_\beta(C,X)-\expdim G^{r,r-s-1}_{d,k}(C).
\end{equation*}
The claim now follows from Proposition \ref{prop:G_expdim}.
\end{proof}

Alternatively, we will see that $\cM^\circ_\beta(C,X)$ is an open subset of the moduli space $Q^+_\beta(C,X)$ constructed in the next section. $Q^+_\beta(C,X)$ will manifestly be pure of expected dimension, also implying Corollary \ref{cor:maps_expdim}. $\cM^\circ_\beta(C,X)$ may be empty, even if $\cM_\beta(C,X)$ is not. For example, $\cM^\circ_\beta(C,X)$ and $G^{r,r-s-1}_{d,k}(C)$ are empty if $d<r$, but $\cM_\beta(C,X)$ always contains degenerate maps if $\beta=d\H^\vee$.

We refer the reader to \cite{chz} and further references therein for an extensive study of the various aspects of the geometry of $G^{r,r-s-1}_{d,k}(C)$, including partial enumerative results when its dimension equals zero. One question to which we do not know the answer is whether $G^{r,r-s-1}_{d,k}(C)$ is always smooth.

If $s=r-2$, then a non-degenerate map $f:C\to \Bl_{\P^{r-2}}(\P^r)$ is given by the data of a degree $d-k$ cover $g\circ f:C\to\P^1$ with underlying line bundle $\cM$, along with an effective divisor $D$ for which $\cM(D)$ has $r+1$ independent sections, including the two defining $g\circ f$. This is similar to the setting of Hurwitz-Brill-Noether theory \cite{jr,HLar2021,llv}, except that in the Hurwitz-Brill-Noether setting, the cover of $\P^1$ is taken to be fixed and general, whereas we take $C$ to be general. The two settings overlap only when $\rho(g,1,d-k)=0$, in which case a general curve admits finitely many covers of $\P^1$ of degree $d-k$.

\subsection{Augmented quasimaps}\label{sec:augmented}

We now compactify $\cM^\circ_\beta(C,X)$ by considering quasimaps relative to inclusions of linear series.

\begin{definition}
Let $\cV\subset \cW$ be the tautological inclusion of universal bundles over $G^{r,r-s-1}_{d,k}(C)$. The moduli space $Q^+_\beta(C,X)$ of \emph{augmented quasimaps} in class $\beta$ is defined by
\begin{equation*}
\pi:Q^+_\beta(C,X)=\P(\cO_{\P(\cV^{r-s})}(-1)\oplus \cW^{s+1}) \to \P(\cV^{r-s})\to G^{r,r-s-1}_{d,k}(C).
\end{equation*}
\end{definition}

$Q^+_\beta(C,X)$ is projective of dimension $\expdim\cM_\beta(C,X)$, by Proposition \ref{prop:G_expdim}. There is a forgetful morphism $\psi:Q^+_\beta(C,X)\to \barQ_\beta(C,X)$ defined as follows. Recall that we have non-vanishing sections $1_D\in H^0(C,\cO(D))$ defined globally over $G^{r,r-s-1}_{d,k}(C)$. A point $u\in Q^+_\beta(C,X)$ with underlying ILS $(\cL,D,V\subset W)$ is given by:
\begin{itemize}
    \item an $(r-s)$-tuple $u_1,\ldots, u_{r-s}\in V\subset H^0(C,\cL(-D))$, parametrized by $\P(\cV^{r-s})$, and
    \item a multiple $u_0=\lambda\cdot 1_D\in H^0(C,\cO(D))$ of the defining section, and sections $u_{r-s+1},\ldots, \\ u_{r+1}\in W\subset H^0(C,\cL)$, parametrized by $\P(\cO_{\P(\cV^{r-s})}(-1)\oplus \cW^{s+1})$.
\end{itemize}
Indeed, given the data of $(u_1,\ldots, u_{r-s})\in V^{r-s}$, we may view a local section of $\cO_{\P(V^{r-s})}(-1)\subset V^{r-s}$ as a scalar multiple $\lambda(u_1,\cdots u_{r-s})\in V^{r-s}$. The scalar $\lambda$ is used to define the section $u_0=\lambda\cdot 1_D\in H^0(C,\cO(D))$, and the data of $u_0,\ldots,u_{r+1}$ define a quasimap $\psi(u)$ with underlying line bundles $\cL_d=\cL$ and $\cL_k=\cO(D)$. The projective bundle structure gives exactly the conditions (NZ1) and (NZ2), as well as the equivalence under the $(\C^{*})^2$-action of Data \ref{data:map_to_X}. Explicitly, scaling $u_1,\ldots,u_{r-s}$ by $\lambda_2$ correspondingly scales $\lambda$ (and hence $u_0$) by $\lambda_2^{-1}$, or equivalently $u_{r-s+1},\ldots,u_{r+1}$ by $\lambda_2$. The scaling action defining $\P(\cO_{\P(\cV^{r-s})}(-1)\oplus \cW^{s+1})$ scales $u_0,u_{r-s+1},\ldots,u_{r+1}$ by $\lambda_1$.

Functorially, $\psi$ is defined such that the pullbacks of the tautological bundles are given by
\begin{align*}
    (1,\psi)^{*}\barP_d&=\cP_d\otimes\cO_{Q^+}(1),\\
    (1,\psi)^{*}\barP_k&=\cO(\cD)\otimes\cO_{Q^+}(1)\otimes\cO_{\P(\cV^{r-s})}(-1),\\
    (1,\psi)^{*}(\barP_d\otimes\barP_k^{-1})&=\cP_d(-\cD)\otimes\cO_{\P(\cV^{r-s})}(1),
\end{align*}
and the pullbacks of the tautological sections are given by
\begin{align}\label{eq:pullback_taut_section}
\nonumber (u_1, \ldots, u_{r-s})=& \psi^{*}(\overline{u}_1,\ldots,\overline{u}_{r-s}):\cO_{C\times Q^+_\beta(C,X)}\to (\cP_d(-\cD)\otimes\nu^{*}\cO_{\P(\cV^{r-s})}(1))^{r-s},\\
\nonumber u_0=&\psi^{*}\overline{u}_{0}:\cO_{C\times Q^+_\beta(C,X)} \to \cO(\cD)\otimes\cO_{Q^+}(1)\otimes\cO_{\P(\cV^{r-s})}(-1),\\
(u_{r-s+1}, \ldots, u_{r+1})=&    \psi^{*}(\overline{u}_{r-s+1},\ldots,\overline{u}_{r+1}):\cO_{C\times Q^+_\beta(C,X)}\to (\cP_d\otimes\cO_{Q^+}(1))^{s+1}.
\end{align}
The maps \eqref{eq:pullback_taut_section} are obtained from the universal lines coming from the projective bundle structure on $Q^+_\beta(C,X)$.

The discrepancy between the line bundle $(1,\psi)^{*}\barP_d$ pulled back from $C\times\barQ_\beta(C,X)$ and the Poincar\'{e} bundle $\cP_d$ pulled back from $C\times \Pic^d(C)$ arises from the fact that $\cP_d$ is assumed trivialized along a choice of point of $C$, but $\barP_d$ is not.

The map $\psi$ is an isomorphism over the open locus $\cM_\beta(C,X)^\circ$ of non-degenerate maps, and is compatible with the map $\pi:\cM^\circ_\beta(C,X)\to G^{r,r-s-1}_{d,k}(C)$. We allow ourselves to use the letter $\pi$ for either map, and identify $\cM^\circ_\beta(C,X)$ with its pre-image in $Q^+_\beta(C,X)$. We say that $u$ is bpf if and only if its underlying quasimap $\psi(u)$ is. We will frequently abuse terminology slightly, referring to the data of an augmented quasimap $u$ by its underlying quasimap $\psi(u)$ or the sections $u_j$ underlying $\psi(u)$, when the underlying ILS is implicit.

\begin{definition}\label{def:birank}
The \emph{bi-rank} of $u\in Q_\beta^+(C,X)$ is the tuple of positive integers
\begin{equation*}
    (r_1,r_2)=(\dim \langle u_1,\ldots,u_{r-s} \rangle, \dim \langle  u_1,\ldots,u_{r+1}\rangle).
\end{equation*}
\end{definition}

We have $1\le r_1\le r-s$ and $r_1\le r_2\le r_1+(s+1)$. Note that we have written here $u_1,\ldots,u_{r-s}$ in both arguments as opposed to $u_0u_1,\ldots,u_0u_{r-s}$, so the bi-rank does not depend on whether $u_0$ is zero. Implicitly, in the second argument, the sections $u_1,\ldots,u_{r-s}$ are viewed as sections of $H^0(C,\cL)$ after twisting by $1_D$. The open subset $\cM^\circ_\beta(C,X)$ is the locus of bpf $u$ of bi-rank $(r-s,r+1)$ (the largest possible), where in addition $u_0\neq0$.

We have therefore realized the moduli space of \emph{non-degenerate} maps $f:C\to X$ as an open subset of a proper moduli space $Q^+_\beta(C,X)$ of expected dimension, though we have not yet proven that $\cM^\circ_\beta(C,X)$ is dense (see Corollary \ref{cor:maps_dense}). By Example \ref{eg:bad_components}, it was necessary to throw out degenerate maps. From the point of view of counting curves, this cost is mild: the curves of interest are usually automatically non-degenerate. For example, when $n\ge r+1$, a curve passing through $n$ general points in $X$ must be non-degenerate, so the geometric Tevelev degree $\Tev^X_{g,n,\beta}$ enumerates only non-degenerate curves.

\subsection{Twisting}\label{sec:twisting}

\begin{definition}
    Let $G^{r,r-s-1}_{d,k}(C)^\circ\subset G^{r,r-s-1}_{d,k}(C)$ be the image of $\cM^\circ_\beta(C,X)$ under $\pi:Q^+_\beta(C,X)\to G^{r,r-s-1}_{d,k}(C)$. We say that an ILS is \emph{bpf} if it lies in $G^{r,r-s-1}_{d,k}(C)^\circ$.
\end{definition}

That is, an ILS is bpf if it underlies some non-degenerate map $f:C\to X$ in class $\beta$. Explicitly, an ILS $(\cL,D,V\subset W)$ is bpf if and only if it satisfies the following two properties.
    \begin{enumerate}
        \item[(BPF1)] The map $V\subset H^0(C,\cL(-D))\to H^0(C,\cL(-D)|_p)$ is surjective for every $p\in C$,
        \item[(BPF2)] The map $W \to H^0(C,\cL|_p)$ is surjective for every $p\in\Supp(D)$.
    \end{enumerate}
``Surjective'' in both cases is equivalent to ``non-zero.'' (BPF1) amounts to the statement that not every section of $V$ vanishes at $p$ \emph{as a section of $H^0(C,\cL(-D))$}. If $p\in \Supp(D)$, it is automatic by definition that every section of $V$ vanishes at $p$ \emph{as a section of $H^0(C,\cL)$}. Throughout the rest of the paper, when we speak of a section of $V$ vanishing at $p$, we will mean (unless otherwise noted) as a section of $H^0(C,\cL(-D))$, whereas when we speak of a section of $W$ vanishing at $p$, we will mean as a section of $H^0(C,\cL)$ (as we must). We will often speak of an ILS satisfying (BPF1) or (BPF2) ``at $p$,'' by which we mean we test the surjectivity of the maps in question only at that particular point.

The conditions above are parallel to the properties (BPF1) and (BPF2) that quasimaps need to satisfy in order to be honest maps. We allow ourselves to refer to these properties freely in both contexts. Indeed, $\uW$ is bpf if and only if, for \emph{any} choice of bases
\begin{equation*}
    \langle u_1,\ldots,u_{r-s}\rangle \subset \langle u_1,\ldots,u_{r+1}\rangle
\end{equation*}
of $V\subset W$, the quasimap $u:C\to X$ determined by the $u_j$ and $u_0=1_D$ is bpf. If $u\in Q^+_\beta(C,X)$ is bpf (at $p$), then $\pi(u)$ must also be, but not vice versa.

We will need an explicit way to ``twist down'' base-points of ILS.

\begin{definition}\label{def:twisting}
Let $(\cL,D,V\subset W)\in G^{r,r-s-1}_{d,k}(C)$ be an ILS. Define the operations:
\begin{enumerate}
    \item[(T1)] If every section of $V$ vanishes at $p$, then replace $D$ with $D+p$, leaving $\cL,V,W$ the same. We now have $(\cL,D,V\subset W)\in G^{r,r-s-1}_{d,k+1}(C)$.
    
    \item[(T2)] If every section of $W$ vanishes at $p\in \Supp(D)$, then replace $D$ with $D-p$ and $\cL$ with $\cL(-p)$. The subspaces $V\subset W$ are identified with subspaces of the new $H^0(C,\cL)$ in the obvious way. We now have $(\cL,D,V\subset W)\in G^{r,r-s-1}_{d-1,k-1}(C)$.   
\end{enumerate}
\end{definition}

In fact, both operations also make sense at the level of augmented quasimaps. If $u\in Q^+_\beta(C,X)$, then applying either (T1) or (T2) at $\pi(u)\in G^{r,r-s-1}_{d,k}(C)$ also twists the sections underlying $u$ up or down by a defining section $1_p$, so one obtains, in a canonical way (independent of choice of $1_p$), an augmented quasimap given by the same sections and the twisted ILS. 

\subsection{Incidence loci}\label{sec:incidence_loci}

Let $\Lambda\subset X$ be a linear subspace, by which we mean the proper transform under $b:X\to\P^r$ of a linear subspace $\Lambda_2\subset\P^r$, not contained in the blown-up locus $P\subset\P^r$. We are interested in imposing conditions of the form $f(p)\in \Lambda$ on maps.

Let $\Lambda_1=g(\Lambda)$ be the image of $\Lambda$ under the projection $g:X\to\P^{r-s-1}$. Write $(\ell_1,\ell_2)=(\dim(\Lambda_1),\dim(\Lambda_2))$. The pair $(\ell_1,\ell_2)$, which we call the \emph{bi-dimension} of $\Lambda$, determines the homology class of $\Lambda\subset X$. Writing 
\begin{equation*}
    \alpha=[\alpha_0\alpha_1:\cdots:\alpha_0\alpha_{r-s}:\alpha_{r-s+1}:\cdots:\alpha_{r+1}],
\end{equation*}
$\Lambda\subset X$ is cut out by $(r-s-1)-\ell_1$ linear equations in $\alpha_1,\ldots,\alpha_{r-s}$, and $(s+1+\ell_1)-\ell_2$ \emph{additional} linear equations in $\alpha_0\alpha_1,\ldots,\alpha_{r+1}$. 

Intrinsically, $\Lambda_1\subset\P^{r-s-1}$ is the vanishing locus of the composite morphism
\begin{equation}\label{eq:seq_lambda1}
    \xymatrix{
    \cO_{\P^{r-s-1}}(-1) \ar[r]^(0.42){\iota_1} & H^0(X,\cO_X(\H-\E))^{\vee} \ar[r]^(0.6){\theta_1}  & \C^{(r-s-1)-\ell_1}
    }
\end{equation}
where $\theta_1$ is a surjection of vector spaces; $\Lambda_1$ is identified with the projectivization of $\ker(\theta_1)$. Similarly, $\Lambda_2\subset\P^{r}$ is the vanishing locus of the composite morphism
\begin{equation}\label{eq:seq_lambda2}
\xymatrix{
    \cO_{\P^{r}}(-1) \ar[r]^(0.42){\iota_2} & H^0(X,\cO_X(\H))^{\vee} \ar[r]^(0.66){\theta_2} & \C^{r-\ell_2} 
    }
\end{equation}
and is identified with the projectivization of the kernel of the surjection $\theta_2$. If the $\Lambda_j$ are the images of $\Lambda\subset X$ under $g,b$, respectively, then the pullbacks of \eqref{eq:seq_lambda1} and \eqref{eq:seq_lambda2} to $X$ are related by a commutative diagram
\begin{equation}\label{linear_spaces}
\xymatrix{
    g^{*}\cO_{\P^{r-s-1}}(-1) \ar[r]^(0.42){\iota_1} & H^0(X,\cO_X(\H-\E))^{\vee} \ar[r]^(0.6){\theta_1}  & \C^{(r-s-1)-\ell_1} \\
    b^{*}\cO_{\P^{r}}(-1) \ar[r]^(0.42){\iota_2} & H^0(X,\cO_X(\H))^{\vee} \ar[r]^(0.6){\theta_2} \ar[u]^{1_E^{*}} & \C^{r-\ell_2} \ar[u]^{\rho}
    }
    .
\end{equation}
Then, $\Lambda\subset X$ is the common vanishing locus of both rows of \eqref{linear_spaces}.

Note further that, if the top row of \eqref{linear_spaces} is zero upon restriction to $\alpha\in X$, then the bottom row is zero at $\alpha$ if and only if the composition of
\begin{equation*}
\xymatrix{
    \theta'_{2}:H^0(X,\cO_X(\H))^{\vee} \ar[r]^(0.68){\theta_2}  & \C^{r-\ell_2} \ar[r]^(0.4){\rho'} & \C^{(s+1+\ell_1)-\ell_2}
    }
\end{equation*}
with $\iota_2$ is zero, where $\rho'$ is a choice of complement to $\rho$, so the restriction of $\rho'$ to $\ker(\rho)$ is an isomorphism. Thus, $\Lambda$ is equivalently the common vanishing locus of the morphisms $\theta_{1}\circ\iota_1$ and $\theta'_{2}\circ\iota_2$, and with respect to these defining morphisms has the expected codimension.

Now, let $p\in C$ be a point. Correspondingly, if
\begin{equation*}
    u=[u_0u_1:\cdots:u_0u_{r-s}:u_{r-s+1}:\cdots:u_{r+1}] \in Q^+_\beta(C,X)
\end{equation*}
is bpf, then the corresponding map $f:C\to X$ has $f(p)\in \Lambda$ if and only if the maps
\begin{equation}\label{Q+_degen1}
    \cO_{\P(\cV^{r-s})}(-1)\to \cV^{r-s} \to \cV^{(r-s-1)-\ell_1}\to \nu_{*}(\cP_d(-\cD)|_p)^{(r-s-1)-\ell_1}
\end{equation}
\begin{equation}\label{Q+_degen2}
    \cO_{Q^+}(-1)\to \cO_{\P(\cV^{r-s})}(-1)\oplus \cW^{s+1} \to \cW^{r+1} \to \nu_{*}(\cP_d|_p)^{(s+1+\ell_1)-\ell_2}
\end{equation}
are zero at $u$. In \eqref{Q+_degen1}, the map $\cV^{r-s} \to \cV^{(r-s-1)-\ell_1}$ is $\theta_1\otimes\cV$, where $\cV^{r-s}$ is identified with $H^0(X,\cO(\H-\E))^\vee \otimes \cV$ via the dual basis to $y_1,\ldots,y_{r-s}$ (\S\ref{sec:maps}). Recalling that $\cP_d$ is the Poincar\'{e} line bundle of degree $d$ on $C\times Q^+_\beta(C,X)$ and that $\nu:C\times Q^+_\beta(C,X)\to Q^+_\beta(C,X)$ is the projection, the rightmost map globalizes composition of the inclusion $V\subset H^0(C,\cL(-D))$ and evaluation at $p$. In \eqref{Q+_degen2}, the first map is the  tautological inclusion on $Q^+$, and the second factors through the pullback of the tautological inclusion on $\P(\cV^{r-s})$. The last map is composition of $\theta'_{2}\otimes\cW$ with evaluation at $p$. Both \eqref{Q+_degen1} and \eqref{Q+_degen2} make sense on all of $Q^+_\beta(C,X)$.
\begin{definition}
    Let $\Lambda\subset X$ be a linear subspace as above. We define the \emph{incidence locus} $\Inc^+(p,\Lambda)\subset Q^+_\beta(C,X)$  by the vanishing of the tautological maps \eqref{Q+_degen1} and \eqref{Q+_degen2}.
\end{definition}

\begin{proposition}\label{prop:inc_exp_dim}
    Let $p\in C$ be a general point and let $\Lambda\subset X$ be a linear space of bi-dimension $(\ell_1,\ell_2)$. Then, the incidence locus $\Inc^+(p,\Lambda)\subset Q^+_\beta(C,X)$ is the closure of the locus of non-degenerate maps in $\cM^\circ_\beta(C,X)$ for which $f(p)\in \Lambda$. In particular, $\Inc^+(p,\Lambda)\subset Q^+_\beta(C,X)$ is pure of codimension $r-\ell_2$, the expected.
\end{proposition}

\begin{proof}
We allow $p\in C$ to vary, and form the relative incidence locus $\cI(\Lambda)\subset C\times Q^+_\beta(C,X)$ whose restriction to any $p\in C$ is $\Inc^+(p,\Lambda)$. By its definition as a degeneracy locus, we have $\dim(\cI(\Lambda))\ge \dim(C\times Q^+_\beta(C,X))-(r-\ell_2)$. Let $\cJ\subset \cI(\Lambda)$ be an irreducible component, and let $(p,u)\in \cJ$ be a general point. It is enough to show that $(p,u)$ has the property that $u\in\cM^\circ_\beta(C,X)$, and that $\cJ$ has expected codimension; the Proposition will follow from restricting to a general point of $C$.

We first show that $\pi(u)$ is BPF at $p$. If not, then we apply operations (T1), then (T2) of Definition \ref{def:twisting} to $\pi(u)$ and $u$ at $p$, if applicable. There are three cases:
\begin{enumerate}
    \item (T1) is applied, but (T2) is not.
    \item (T2) is applied, but (T1) is not.
    \item Both operations are applied.
\end{enumerate}

Consider case (1). Then, (T1) may be applied on all of $\cJ$, and defines an injective map $\epsilon:\cJ\to C\times Q^+_{\beta'}(C,X)$, where $\beta'=d\H^\vee+(k+1)\E^\vee$. Let $(p,u')$ be the image of a general point of $\cI$ under $\epsilon$, and write $\uW'=(\cL',D',V'\subset W'),u'_j$ for the data underlying $u'$. Note first that we have the condition on $p$ that $p\in \Supp (D')$. Next, by assumption, (BPF2) holds for $\uW'$ at $p$, so not all possible sections $u'_j\in W$ vanish at $p$. Therefore, if $p,\pi(u')$ are fixed, then the $(s+1+\ell_1)-\ell_2$ linear equations that must be satisfied by the sections $u'_{r-s+1},\ldots,u'_{r+1}$ upon evaluation at $p$ impose the expected number of conditions on $u'$. It follows that the dimension of the image under $\epsilon$ of a neighborhood of a general point $(p,u)\in \cJ$ is at most
\begin{equation*}
 \dim(C\times Q^+_{\beta'}(C,X))-1-((s+1+\ell_1)-\ell_2)<\dim(C\times Q^+_{\beta}(C,X))-(r-\ell_2)\le \dim(\cJ),
\end{equation*}
a contradiction.

Case (2) is similar. (T2) defines an injective map $\epsilon:\cJ\to C\times Q^+_{(d-1)\H^\vee+(k-1)\E^\vee}(C,X)$, but the expected number $(r-s-1)-\ell_1$ of conditions are imposed on the twisted sections $u'_1,\ldots,u'_{r-s}$ in the image; comparing dimensions yields a contradiction. Case (3) is easier: (T1) and (T2) together define an injective map $\epsilon:\cJ\to C\times Q^+_{(d-1)\H^\vee+k\E^\vee}(C,X)$, but the target has dimension strictly less than $\dim(\cJ)$.

We may repeat the above arguments to show that $\pi(u)$ is in fact BPF at all other points. Indeed, after operation (T1) or (T2) is applied at $q\neq p$ for a $(p,u)\in \cJ$, all linear conditions on the twisted sections $u'_j$ evaluated at $p$ persist. In case (1), one may define a map $\epsilon:\cJ^\circ\to C\times Q^+_{d\H^\vee+(k+1)\E^\vee}$ in a neighborhood\footnote{$\epsilon$ is only defined locally on $\cJ$, because it may happen that $q$ becomes equal to $p$ on a closed subvariety of $\cJ$. In fact, one should take $\cJ^\circ$ to be an \emph{analytic} open (or \'{e}tale) neighborhood $i:\cJ^\circ\to \cJ$ of $(p,u)$. Indeed, $\epsilon$ is defined more precisely as follows: let $s:\cJ^\circ\to C$ be a map such that, for all $(p^\bullet,u^\bullet)\in\cJ^\circ$, the augmented quasimap $i(u^\bullet)$ fails to satisfy (BPF1) at $s((p^\bullet,u^\bullet))$, and furthermore $s((p,u))=q$. Then, $\epsilon((p^\bullet,u^\bullet))$ is given by applying (T1) to $i(u^\bullet)$ at $s((p^\bullet,u^\bullet))$. The role of $s$ is to identify a consistent choice of twisted point in a neighborhood of $(p,u)$. If there are multiple points $q\neq p$ for which $u$ does not satisfy (BPF1) at $q$, then such a map $s$ may not exist Zariski-locally, but does in the analytic or \'{e}tale topologies. The dimension counts remain unchanged.} $\cJ^\circ$ of $(p,u)\in \cJ$ by applying operation (T1) at $q\neq p$. The map $\epsilon$ remains finite (but may no longer be injective), with fibers corresponding to choices of points in the support of the twisted divisor $D$, and there are at least as many conditions on $\im(\epsilon)$ as in the case of twisting at $p$. In case (2), the map $\epsilon$ has 1-dimensional fibers, corresponding to choices of points on $C$, but there are strictly more conditions on $\im(\epsilon)$ than in the case of twisting at $p$. Case (3) is similar. 

Next, we claim the \emph{quasimap} $u$ is BPF at $p$. Indeed, $\pi(u)$ is BPF at $p$, so not all sections in $V,W$, from which the $u_j$ must be chosen, vanish at $p$. If (BPF1), but not (BPF2), is to fail at $p$, then each of the $r-s$ sections $u_1,\ldots,u_{r-s}$ must vanish at $p$, which imposes $r-s$ independent conditions, in addition to the $(s+1+\ell_1)-\ell_2$ conditions imposed on the remaining sections. This gives strictly more than $r-\ell_2$ conditions on the fiber over $\pi(u)$, which is a contradiction. Suppose instead that (BPF2), but not (BPF1), fails at $p$, and furthermore that $p\notin \Supp(D)$. Then, in addition to the $r-s-\ell_1$ independent conditions on $u_1,\ldots,u_{r-s}$, we have that $u_{r-s+1},\ldots,u_{r+1}$ must vanish at $p$ and that $u_0=0$, which again imposes too many conditions on the fiber over $\pi(u)$. If instead $p\in \Supp(D)$, then lying in the support of $D$ imposes one condition \emph{on $p$}, in addition to $r-\ell_1$ conditions on the $u_j$. Once more, too many conditions are imposed near $(p,u)\in \cJ$.

Similarly, (BPF1) and (BPF2) cannot simultaneously fail at $p$. Repeating the arguments shows that $u$ is in fact BPF everywhere (away from $p$).

Finally, let $(r_1,r_2)$ be the bi-rank of $u$. Let $Q^{(r_1,r_2)}_\beta(C,X)^{\circ}\subset Q^+_\beta(C,X)$ be the locally closed subset consisting of bpf augmented quasimaps of bi-rank $(r_1,r_2)$. Consider the map $\ev:C\times Q^{(r_1,r_2)}_\beta(C,X)^{\circ}\to X$ that evaluates at $p$. Suppose that $\cJ$ is generically contained in $C\times Q^{(r_1,r_2)}_\beta(C,X)^{\circ}$. For any fixed $p\in C$ and bpf $\uW\in G^{r,r-s-1}_{d,k}(C)$, the restriction of $\ev$ to $p\times \pi^{-1}(\underline W)\subset C\times Q^{(r_1,r_2)}_\beta(C,X)^{\circ}$ is easily seen to have equidimensional fibers. Thus, the restriction of $\cJ$ to $C\times Q^{(r_1,r_2)}_\beta(C,X)^{\circ}$ has codimension equal to that of $\Lambda$ in $X$, which is $r-\ell_2$. It follows $u$ must be of full bi-rank, and $\cJ$ must be of codimension exactly $r-\ell_2$. On the locus on $Q^{(r-s,r+1)}_\beta(C,X)^{\circ}$ where $u_0$ is identically zero, we see that $\Inc(p,\Lambda)$ also has codimension $r-\ell_2$, so a general point $u\in \Inc(p,\Lambda)$ must in fact lie in $\cM^\circ_\beta(C,X)$. 
\end{proof}

By taking $\Lambda=X$, we obtain:

\begin{corollary}\label{cor:maps_dense}
  The locus of non-degenerate maps $\cM_\beta^\circ(C,X)\subset Q^+_\beta(C,X)$ is dense. 
\end{corollary}

Similarly, $G^{r,r-s-1}_{d,k}(C)^\circ\subset G^{r,r-s-1}_{d,k}(C)$ is dense.

One can now plausibly hope to enumerate non-degenerate maps $f:C\to X$ subject to general incidence conditions of the form $f(p_i)\in \Lambda_i$ by considering the intersection
\begin{equation*}
    \bigcap_{i=1}^{n}\Inc^+(p_i,\Lambda_i)\subset Q^+_{\beta}(C,X).
\end{equation*}
However, as the next example shows, such an intersection can still fail to be transverse.

\begin{example}\label{eg:excess}
    Take $X=\Bl_q(\P^2)$, and write $\beta=d\H^\vee+k\E^\vee=d\H-k\E$. Take the numerical assumption \eqref{dim_constraint} of Definition \ref{def:tev}, which is $n=\frac{1}{2}(3d-k)-g+1$. Fix general points $p_1,\ldots,p_n\in C$ and $x_1,\ldots,x_n\in X$, and consider the intersection
    \begin{equation*}
       I=\bigcap_{i=1}^{n} \Inc^+(p_i,x_i)\subset Q^+_\beta(C,X),
    \end{equation*}
    of expected dimension 0. 
    
    However, suppose that $d-k\ge n$ (equivalently, that $d+k\le 2g-2$), and that there exists $\uW=(\cL,D,V\subset W)$ with the property that $V$ contains a non-zero section $v$ vanishing (as a section of $\cL(-D)$) at all of $p_1,\ldots,p_n$. Then, any augmented quasimap
    \begin{equation*}
        [u_0u_1:u_0u_2:u_3]=[(\alpha_0\cdot 1_D)(\alpha_1 v):(\alpha_0\cdot 1_D)(\alpha_2 v):\alpha_3 v1_D]\in \pi^{-1}(\uW),
    \end{equation*}
    where $[\alpha_0\alpha_1:\alpha_0\alpha_2:\alpha_3]\in X$, lies in $I$.
\end{example}

Such a $\uW$ is expected to exist whenever $\dim (G^{2,0}_{d,k}(C))\ge n-1$, which is equivalent to $n\ge 5$. One case in which such $\uW$ can be easily seen to exist is when $k=1$. Then, by the pointed Brill-Noether existence theorem \cite[Theorem 1.1]{osserman_bn_pointed}, there exist linear series $W\in W^2_d(C)$ with a section $v$ vanishing along $p_1+\cdots+p_n$. Take $D=p_0$ to be a different point on which $v$ vanishes, let $v'$ be another section in $W$ vanishing along $p_0$, and take $V=\langle v,v'\rangle$.

The problem in Example \ref{eg:excess} is that, on the bi-rank $(1,1)$ stratum, each incidence locus $\Inc^+(p_i,x_i)$ imposes 1 condition, whereas 2 are expected. While this poses no issues in Proposition \ref{prop:inc_exp_dim}, the excess intersection arises once there are enough incidence loci imposing too few conditions to overcome the codimension of the low bi-rank stratum. In order to compute enumerative invariants, one should thus blow up the strata of low bi-rank.

\subsection{Complete quasimaps}\label{sec:complete}

We now come to the main definition of this paper.

\begin{definition}\label{def:complete_Q}
    Let $b_Q:\wtQ_\beta(C,X)\to Q^+_\beta(C,X)$ be the iterated blow-up of the closures of the strata on $Q_\beta^+(C,X)$ of augmented quasimaps of bi-rank $(r_1,r_2)$, in order, first by $r_1$, then by $r_2$. The space $\wtQ_\beta(C,X)$ is the moduli space of \emph{complete quasimaps} to $X$ in class $\beta$.
\end{definition}

That is, first blow up the stratum of augmented quasimaps of bi-rank $(1,1)$. Next, blow up the proper transform of the closure of the stratum of augmented quasimaps of bi-rank $(1,2)$.  Continue until reaching bi-rank $(1,s+2)$, then proceed to bi-rank $(2,2)$, and so on. We expect it to be true that all blow-up centers are smooth over $G^{r,r-s-1}_{d,k}(C)$ (Ingredient \ref{prop:wtQ_points}), but only give a proof in dimension 2.

Note that $b_Q$ is an isomorphism over the open locus $\cM^{\circ}_\beta(C,X)\subset Q^+_\beta(C,X)$; we identify $\cM^{\circ}_\beta(C,X)$ with its pre-image in $\wtQ_\beta(C,X)$. We denote, abusing notation, the composition $\pi\circ b_Q:\wtQ_\beta(C,X)\to G^{r,r-s-1}_{d,k}(C)$ also by $\pi$.

\begin{definition}
    Let $\Lambda\subset X$ be a linear subspace as in \S\ref{sec:incidence_loci}, and let $p\in C$ be a point. Let $\Inc(p,
    \Lambda)\subset \wtQ_\beta(C,X)$ be the proper transform under $b_Q$ of $\Inc^+(p,\Lambda)\subset Q^+_\beta(C,X)$.
\end{definition}

By definition, the same conclusions of Proposition \ref{prop:inc_exp_dim} for $\Inc^+(p,\Lambda)\subset Q^+_\beta(C,X)$ also hold for $\Inc(p,\Lambda)\subset \wtQ_\beta(C,X)$. We are now ready to state the main conjecture.

\begin{conjecture}\label{conj:main}
Let $p_1,\ldots,p_n\in C$ be general points, and let $\Lambda_1,\ldots,\Lambda_n\subset X$ be general linear spaces. Then, the intersection
\begin{equation*}
    I=\bigcap_{i=1}^{n}\Inc(p_i,\Lambda_i)\subset \wtQ_\beta(C,X)
\end{equation*}
is generically smooth and pure of codimension $\sum_{i=1}^{n}(r-\dim(\Lambda_i))$, the expected. Furthermore, any general point of $I$ lies in $\cM^\circ_\beta(C,X)$. 
\end{conjecture}

Conjecture \ref{conj:main} implies Conjecture \ref{conj:main_intro}. Indeed, in the setting of Definition \ref{def:tev}, the intersection $I$ would consist of finitely many reduced points, all corresponding to non-degenerate maps $f:C\to X$ satisfying $f(p_i)=x_i$. Under the assumption $n\ge r+1$, any such map is non-degenerate. More generally, if
\begin{equation*}
    \sum_{i=1}^{n}(r-\dim(\Lambda_i))=(r+1)d-(r-s-1)k-r(g-1)=\dim\wtQ_\beta(C,X).
\end{equation*}
then, Conjecture \ref{conj:main} predicts that $I$ is a union of reduced points corresponding to the set of non-degenerate maps $f:C\to X$ in class $\beta$ with $f(p_i)\in \Lambda_i$.

Conjecture \ref{conj:main} is a direct analog of \cite[Theorem 1.7]{lian_pr} for complete quasimaps to $\P^r$ (i.e., complete collineations relative to linear series). The role of the Brill-Noether theorem in transversality for complete quasimaps to $\P^r$ in \cite[Theorem 1.7]{lian_pr} is replaced by Proposition \ref{prop:G_expdim} and Corollary \ref{cor:maps_expdim}. However, some further basic tools in the setting of $X=\Bl_{\P^s}(\P^r)$ are, at the moment, missing. For example, we do not even know if the moduli space of non-degenerate maps $\cM^\circ_\beta(C,X)$ itself is always generically smooth beyond dimension 2, which is the case $n=0$ in Conjecture \ref{conj:main}. The \emph{pointed} Brill-Noether theorem is also crucial in the case $X=\P^r$, see \cite[Lemma 4.4]{lian_pr}, but we do not have a suitable analog for $X=\Bl_{\P^s}(\P^r)$.

However, in dimension 2, these ingredients are afforded by standard results on Severi varieties of surfaces, see Theorem \ref{thm:BN_surface}. We focus on this case in the next two sections, and prove Conjecture \ref{conj:main} for $X=\Bl_q(\P^2)$.

\section{Set-theoretic descriptions}\label{sec:set-theoretic}

In this section, we describe the points of the moduli spaces $\wtQ_\beta(C,X)$ of complete quasimaps and the incidence loci $\Inc(p,\Lambda)$. We focus on the case $X=\Bl_q(\P^2)$, the only one in which Conjecture \ref{conj:main} is proven. The general case is discussed without proofs in \S\ref{sec:higher_dim}.

\subsection{The $(1,1)$ stratum in dimension 2}

We specialize to the case $X=\Bl_q(\P^2)$. In this case, $b_Q:\wtQ_\beta(C,X)\to Q^+_\beta(C,X)$ is a single blow-up, at the bi-rank $(1,1)$ stratum, as the bi-rank $(1,2)$ and $(2,2)$ strata are already Cartier divisors. Thus, describing the points of $\wtQ_\beta(C,X)$ amounts to describing the bi-rank $(1,1)$ stratum and its normal bundle.

For brevity, we write throughout this section $Q^+=Q^+_\beta(C,X)$ and $G_{d,k}=G^{2,0}_{d,k}(C)$, where $\beta=d\H^\vee+k\E^\vee=d\H-k\E$. Recall that, by definition, we have a tower of projective bundles 
\begin{equation*}
    \pi:Q^+=\P(\cO_{\P(\cV^2)}(-1)\oplus\cW)\to \P(\cV^2) \to G_{d,k}.
\end{equation*}
A point of $Q^+$ is denoted $u=[u_0u_1:u_0u_2:u_3]$; the data of the underlying ILS is implicit.

Consider now the stratum $Z$ of augmented quasimaps of bi-rank $(1,1)$. We have a diagram
\begin{equation*}
    \xymatrix{
    Z=\P(\cO(-1,-1)\oplus\cO_{\P(\cV)}(-1))  \ar[r] & \P(\cO(-1,-1)\oplus\cW) \ar[r] \ar[d] & Q^+ \ar[d] \\
     & \P^1\times\P(\cV) \ar[r]^{\sigma} & \P(\cV^2)
    }
\end{equation*}
The map $\sigma$ is the relative Segre embedding, cutting out the locus where $u_1,u_2$ are dependent, and the right square is Cartesian. The pullback $\sigma^{*}\cO_{\P(\cV^2)}(-1)$ is denoted $\cO(-1,-1)=\cO_{\P^1}(-1)\otimes\cO_{\P(\cV)}(-1)$. The locus where $u_3$ lies in the 1-dimensional span of $u_1,u_2$ is cut out by the sub-bundle $\cO_{\P(\cV)}(-1)\subset \cV\subset \cW$. We see immediately:
\begin{proposition}\label{prop:Z_as_product}
    We have $Z\cong X\times \P(\cV)$. In particular, $Z\to G_{d,k}$ is smooth of relative dimension 3. 
\end{proposition}

\begin{proof}
Twisting by $\cO_{\P(\cV)}(1)$, we see in fact that
\begin{equation*}
    Z=\P(\cO(-1,-1)\oplus\cO_{\P(V)}(-1))\cong \P(\cO_{\P^1}(-1)\oplus\cO)\cong X\times \P(\cV) \to \P^1\times\P(\cV).
\end{equation*}
Concretely, a point $\alpha=[\alpha_0\alpha_1:\alpha_0\alpha_2:\alpha_3]\in X$ and a section $v\in \cV\subset H^0(C,\cL(-D))$ correspond to a ``constant quasimap $u:C\to X$ with image $\alpha$.'' In coordinates,
\begin{equation*}
    u=[(\alpha_01_D)\cdot (\alpha_1v):(\alpha_01_D)\cdot (\alpha_2v):\alpha_3v1_D]\in Z.
\end{equation*}
\end{proof}

We now identify the normal bundle $N_{Z/Q^+}$. We have the diagram of sheaves on $X$
\begin{equation}\label{eq:euler}
\xymatrix{
    0 \ar[r]  & \cO \ar[r]^(0.25){A_1} \ar[d] &  \cO(\H-\E)^{\oplus 2}\oplus \cO(\H) \ar[r] \ar[d] &  T_X(-\log E) \ar[r] \ar[d] &  0 \\
    0 \ar[r]  & \cO^{\oplus 2} \ar[r]^(0.25)A &  \cO(\E)\oplus \cO(\H-\E)^{\oplus 2}\oplus \cO(\H) \ar[r] &  T_X \ar[r] &  0 
    }.
\end{equation}
The map $A$ is the direct sum of $A_1$ and $A_2:\cO\to \cO(\E)\oplus\cO(\H)$, which are induced by the tuples of torus-invariant sections $(y_1,y_2,y_3)$ and $(1_E,y_3)$, respectively (see \ref{sec:maps}). The cokernel of $A$ is given in coordinates as follows: at $\alpha=[\alpha_0\alpha_1:\alpha_0\alpha_2:\alpha_3]\in X$, a tangent vector to $\alpha$ is given by
\begin{equation*}
    [(\alpha_0+\epsilon v_0)(\alpha_1+\epsilon v_1):(\alpha_0+\epsilon v_0)(\alpha_2+\epsilon v_2):\alpha_3+\epsilon v_3],
\end{equation*}
where $(v_0,v_1,v_2,v_3)$ may be regarded as a section of $\cO(\E)\oplus \cO(\H-\E)^{\oplus 2}\oplus \cO(\H)$.

The log tangent bundle $T_X(-\log E)$ can be taken by definition to be the cokernel of $A_1$. Its dual is denoted $\Omega_X(\log E)$. The vertical map $T_X(-\log E)\to T_X$ is dual to the kernel of the residue map $\Omega_X(\log E)\to \cO_E$. Pullback of differentials by the projection $g:X\to\P^1$ induces an inclusion of the sub-line bundle $g^{*}\Omega_{\P^1}\subset \Omega_X(\log E)$. In terms of the dual Euler sequence, this sub-bundle is
\begin{equation*}
    \ker(\cO(-(\H-\E))^2\to\cO) \hookrightarrow \ker(\cO(-(\H-\E))^2\oplus\cO(-\H)\to\cO)
\end{equation*}
induced by the inclusion of the first two coordinates.

\begin{proposition}\label{normal_bundle_canonical}
    As bundles on $Z\cong X\times\P(\cV)$, we have
    \begin{equation*}
        N_{Z/Q^+}\cong \cHom_1(\Omega_X(\log E),\cW/\cO_{\P(\cV)}(-1) \otimes \cO_{\P(\cV)}(1)),
    \end{equation*}
     where the right hand side is by definition the kernel of
    \begin{equation}\label{eqn: map defining N}
    \cHom(\Omega_X(\log E),\cW/\cO_{\P(\cV)}(-1) \otimes \cO_{\P(\cV)}(1)) \to \cHom(g^* \Omega_{\P^1}, \cW/\cV \otimes \cO_{\P(\cV)}(1)).
    \end{equation}
\end{proposition}

That is, at the point $(\alpha,v)\in X\times \P(\cV)$, the normal bundle to $Z$ is identified with the subspace of $\Hom(\Omega_X(\log E)|_\alpha, W/v)$ where $g^* \Omega_{\P^1}|_\alpha$ maps to $V/v$. This identification is easily described pointwise in coordinates; the proof of Proposition \ref{normal_bundle_canonical} is merely a globalization of what follows. Relatively to $G_{d,k}$, a tangent vector of $Q^+$ at $(\alpha,v)$ is given by
\begin{equation}\label{eqn: normal vector}
    [(\alpha_0 1_D+\epsilon v_0)(\alpha_1v+\epsilon v_1):(\alpha_0 1_D+\epsilon v_0)(\alpha_2v+\epsilon v_2):\alpha_3v1_D+\epsilon v_3],
\end{equation}
where $v_0\in H^0(C,\cO(D))$ is a multiple of $1_D$, we have $v_1,v_2\in V\subset H^0(C,\cL(-D))$, and $v_3\in W$. The deformation of $\alpha_01_D$ by $v_0$ is along $Z$, so may safely be ignored. 

The triple $(v_1,v_2,v_3)$ gives rise to a linear map
\begin{equation*}
    \widetilde{\phi}:(\cO(-(\H-\E))^2\oplus\cO(-\H))|_\alpha\to W/v,
\end{equation*}
defined as follows. First, define the non-zero vector $e_1$ in the first summand $\cO(-(\H-\E))|_{\alpha}$ above to be the vector sent to $ (\alpha_1,\alpha_2)\in \cO_X|_\alpha^2$ under the embedding $\cO(-(\H-\E))|_{\alpha}\to \cO^2|_\alpha$. Define $e_2$ to be the same vector in the second summand $\cO(-(\H-\E))|_{\alpha}$. Define $e_3\in \cO(-\H)|_{\alpha}$ to be the vector sent to $( \alpha_0\alpha_1, \alpha_0\alpha_2, \alpha_3)$ under $\cO(-\H)|_{\alpha}\to \cO^3|_\alpha$. Finally, define $\widetilde{\phi}(e_j)$ to be the image of $v_j$ in $W/v$ (or more precisely, $1_Dv_j$, if $j=1,2$). We are free to quotient by $v$ because deformations by multiples of $v$ are also along $Z$. Restricting to $\Omega_X(\log E)|_\alpha$ yields an element of $\Hom_1(\Omega_X(\log E)|_\alpha, W/v)$ because $v_1,v_2\in V$.

\begin{proof}[Proof of Proposition \ref{normal_bundle_canonical}]
Write $\cO_{Q^+}(1)$ for the relative $\cO(1)$ of the $\P^3$-bundle $Q^+\to \P(\cV^2)$. Then, we have an Euler sequence for $T_{Q^+/G_{d,k}}$
\begin{equation*}
   0\to \cO^{\oplus 2}\to  (\cO_{\P(\cV^2)}(-1)\otimes \cO_{Q^+}(1)) \oplus (\cV\otimes \cO_{\P(\cV^2)}(1))^{\oplus 2} \oplus (\cW\otimes \cO_{Q^+}(1)) \to T_{Q^+/G_{d,k}} \to 0
\end{equation*}
whose pullback to $Z=X\times\P(\cV)$ is
\begin{align}\label{euler_Z}
\nonumber
   0\to \cO^{\oplus 2} &\to  (\cO(\E)\otimes \cO_{\P(\cV)}) \oplus (\cO(\H-\E)\otimes(\cV\otimes \cO_{\P(\cV)}(1)))^{\oplus 2} \oplus (\cO(\H)\otimes (\cW\otimes \cO_{\P(\cV)}(1)))\\ 
   &\to T_{Q^+/G_{d,k}}|_{X\times\P(\cV)} \to 0.
\end{align}
The three summands in the middle term parametrize choices of $v_0,(v_1,v_2),v_3$ in \eqref{eqn: normal vector}.

The injection on the left of \eqref{euler_Z} is given by the direct sum of the maps
\begin{align*}
    \cO &\to (\cO(\E)\otimes \cO_{\P(\cV)}) \oplus (\cO(\H)\otimes (\cW\otimes \cO_{\P(\cV)}(1))),\\
    \cO &\to (\cO(\H-\E)\otimes(\cV\otimes \cO_{\P(\cV)}(1)))^{\oplus 2} \oplus (\cO(\H)\otimes (\cW\otimes \cO_{\P(\cV)}(1))),
\end{align*}
induced by the inclusions of sub-bundles $\cO_{\P(\cV)}(-1)\subset \cV\subset \cW$, and correspond to trivial deformations in $Q^+$ relative to $G_{d,k}$.

To obtain $N_{Z/Q^+}$, we need to quotient further by tangent directions along $Z=X\times\P(\cV)$. Tangent vectors to $X$ are those in the image of 
\begin{equation*}
    (\cO(\E)\otimes \cO_{\P(\cV)}) \oplus (\cO(\H-\E)\otimes \cO_{\P(\cV)})^{\oplus 2} \oplus (\cO(\H)\otimes \cO_{\P(\cV)}),
\end{equation*}
embedded in the natural way into the middle term of \eqref{euler_Z}. (The quotients by $\cO^{\oplus 2}$ are suppressed.) In the coordinates of \eqref{eqn: normal vector}, these are the tangent vectors for which $v_1,v_2,v_3$ are all multiples of $v$. Relative tangent vectors to $\P(\cV)$ are those in the image of $\cV\otimes \cO_{\P(\cV)}(1)$, which is included in the factors 
\begin{equation*}
    (\cO(\H-\E)\otimes(\cV\otimes \cO_{\P(\cV)}(1)))^2 \oplus (\cO(\H)\otimes (\cW\otimes \cO_{\P(\cV)}(1))
\end{equation*}
in \eqref{euler_Z}. In the coordinates of \eqref{eqn: normal vector}, these are the tangent vectors for which $v_0=0$, and $(v_1,v_2,v_3)$ is equal to $(\alpha_1\overline{v},\alpha_2\overline{v},\alpha_3\overline{v})$ for some $\overline{v}\in \cV \otimes \cO_{\P(\cV)}(1)$ (well defined modulo $v$).

We now define an isomorphism $N_{Z/Q^+}\cong \cHom_1(\Omega_X(\log E),\cW/\cO_{\P(\cV)}(-1) \otimes \cO_{\P(\cV)}(1))$. We have natural maps
\begin{align*} 
    \widetilde{\Phi}:&(\cO(\E)\otimes \cO_{\P(\cV)}) \oplus (\cO(\H-\E)^{\oplus 2 }\otimes(\cV\otimes \cO_{\P(\cV)}(1))) \oplus (\cO(\H)\otimes (\cW\otimes \cO_{\P(\cV)}(1)))\\
    \to & (\cO(\E) \oplus \cO(\H-\E)^{\oplus 2} \oplus \cO(\H)) \otimes (\cW \otimes \cO_{\P(\cV)}(1))\\
    \to & T_X \otimes (\cW/\cO_{\P(\cV)}(-1) \otimes \cO_{\P(\cV)}(1))\\
    =&\cHom(\Omega_X,\cW/\cO_{\P(\cV)}(-1) \otimes \cO_{\P(\cV)}(1)).
\end{align*}

Because $\widetilde{\Phi}$ is zero upon restriction to the first summand $\cO(\E)\otimes \cO_{\P(\cV)}$, it factors through $T_X(-\log E) \otimes (\cW/\cO_{\P(\cV)}(-1) \otimes \cO_{\P(\cV)}(1))$. Moreover, because the maps 
\begin{equation*}
\cO(\H-\E)^{\oplus 2} \otimes (\cV \otimes \cO_{\P(\cV)}(1)) \to g^* T_{\P^1} \otimes (\cW/\cV \otimes \cO_{\P(\cV)}(1))
\end{equation*}
and $\cO(\H) \to g^* T_{\P^1}$ are also zero, it further factors through $\cHom_1(T_X(-\log E),\cW/\cO_{\P(\cV)}(-1) \otimes \cO_{\P(\cV)}(1))$. We therefore obtain a map
\begin{align*} 
    \Phi:&(\cO(\E)\otimes \cO_{\P(\cV)}) \oplus (\cO(\H-\E)\otimes(\cV\otimes \cO_{\P(\cV)}(1)))^{\oplus 2} \oplus (\cO(\H)\otimes (\cW\otimes \cO_{\P(\cV)}(1)))\\
    \to &\cHom_1(\Omega_X(\log E),\cW/\cO_{\P(\cV)}(-1) \otimes \cO_{\P(\cV)}(1)).
\end{align*}

We claim that $\Phi$ is surjective. Indeed, after tensoring with $\cO_{\P(\cV)}(1)$, the image of
 \begin{equation*}
   \Psi: (T_X(-\log E) \otimes \cV / \cO_{\P(\cV)}(-1) ) \oplus (\cO(\H) \otimes \cW/\cO_{\P(\cV)}(-1)) \to T_X(-\log E) \otimes \cW/\cO_{\P(\cV)}(-1) )
 \end{equation*}
is equal to the kernel of the morphism \eqref{eqn: map defining N}. Restricting $\Phi$ separately to $(\cO(\H-\E)^2\oplus\cO(\H))\otimes(\cV\otimes \cO_{\P(\cV)}(1))$ and $(\cO(\H)\otimes (\cW\otimes \cO_{\P(\cV)}(1)))$ gives the entire image of $\Psi$.

Finally, we claim that the sub-bundles of the domain of $\Phi$ by which we quotient to obtain $N_{Z/Q^+}$ are all in the kernel. Indeed, the sub-bundle $\cO^{\oplus 2}\subset (\cO(\E)\otimes \cO_{\P(\cV)}) \oplus (\cO(\H-\E)\otimes(\cV\otimes \cO_{\P(\cV)}(1)))^2 \oplus (\cO(\H)\otimes (\cW\otimes \cO_{\P(\cV)}(1)))$ maps under $\Phi$ to zero in the first factor of $T_X \otimes (\cW/\cO_{\P(\cV)}(-1) \otimes \cO_{\P(\cV)}(1))$. Moreover, tangent vectors to $Z$ in the $\P(\cV)$- and $X$-directions map under $\Phi$ to zero in the first and second factors, respectively. Thus, $\Phi$ descends to a map
\begin{equation}\label{eqn: iso Phi0}
\Phi_0:N_{Z/Q^+}\to \cHom_1(\Omega_X(\log E),(\cW/\cO_{\P(\cV)}(-1) \otimes \cO_{\P(\cV)}(1))),
\end{equation}
which is a surjection of vector bundles of the same rank, and therefore an isomorphism.
\end{proof}

We conclude that a point of the exceptional divisor $E_\beta(C,X)\subset \wtQ_\beta(C,X)$ is given by the data of $(\alpha,v)\in X\times \P(\cV)$ as in Proposition \ref{prop:Z_as_product} and a linear map
\begin{equation*}
    \phi\in \P\Hom_1(\Omega_X(\log E)|_\alpha, W/v).
\end{equation*}

\subsection{Incidence loci in dimension 2}

We continue to assume $X=\Bl_q(\P^2)$. We wish now to give set-theoretic descriptions of the incidence loci $\Inc(p,\Lambda)\subset\wtQ_\beta(C,X)$. We view points of $\Inc(p,\Lambda)$ as limits of 1-parameter families inside $\Inc^+(p,\Lambda)$ approaching $Z$.

\begin{proposition}\label{prop:inc_lines}
Let $\Lambda\subset X$ be a line, either of class $\H$ or $\H-\E$\footnote{Lines of class $\H$ have bi-dimension $(1,1)$, and lines of class $\H-\E$ have bi-dimension $(0,1)$.}. Then, the proper transform $\Inc(p,\Lambda)$ of $\Inc^+(p,\Lambda)$ under $b:\wtQ_\beta(C,X)\to Q^+_\beta(C,X)$ is equal to its pullback.
\end{proposition}

\begin{proof}
    Immediate from the observation that, upon restriction to any fiber of $\pi:Q^+\to G_{d,k}$, either $\Inc^+(p,\Lambda)$ and $Z$ intersect transversely, or $\Inc^+(p,\Lambda)$ contains the entire fiber.
\end{proof}

The more interesting case is when $\Lambda=x\in X$ is a point. We assume that 
\begin{equation*}
    x=[1\cdot \gamma_1:1\cdot \gamma_2:\gamma_3]
\end{equation*}
lies in $X^\circ:=X-E$. Write again $g:X\to\P^1$ for the projection.

\begin{proposition}\label{prop:inc_exceptional}
Let $u=((\alpha,v),\phi)\in E_\beta(C,X)$ be a point. Let $p\in C$ be a general point, and let $x\in X^\circ$ be any point. For any $\alpha\neq x$, there exists a line $\ell_{x,\alpha}\subset \Omega_X(\log E)|_{\alpha}$, depending only on $x$ and $\alpha$ (so not on $C,p,\beta$), such that $\Inc(p,x)\cap E_\beta(C,X)$ admits the following characterization.

We have $u\in \Inc(p,x)$ if and only if one of the following holds: 
\begin{enumerate}
    \item[(i)] $\alpha=x$, or
    \item[(ii)] $g(\alpha)=g(x)$ and $p\in\Supp(D)$, or
    \item[(iii)] $v$ vanishes, as a section of $\cL(-D)$, at $p$. Furthermore, $\phi(\ell_{x,\alpha})\in W/v$ vanishes at $p$.
    \end{enumerate}
\end{proposition}
In case (iii), it is assumed that $\alpha\neq x$ (otherwise we are in case (i)), so that the line $\ell_{x,\alpha}$ will be well-defined. The last vanishing requirement makes sense exactly because $v$ is required to vanish at $p$. 

The line $\ell_{x,\alpha}$ is defined in coordinates in the proof below, but it can be identified intrinsically as follows. First, if $g(\alpha)=g(x)$, then $\ell_{x,\alpha}=(g^{*}\Omega_{\P^1})|_\alpha$. Otherwise, let $h:X\to\P^2\dashrightarrow \P^1$ be the projection from $x$, which contracts $E$ to a point $y\in\P^1$, but is indeterminate at $x$. Then, 
\begin{equation*}
\ell_{x,\alpha}=h^{*}\Omega_{\P^1}(y)|_{h(\alpha)}\subset\Omega_X(\log E)|_{\alpha}.
\end{equation*}
The pre-image $h^{-1}(y)$ is the union of $E$ and the unique line $L'_x\subset X$ through $x$ of class $\H-\E$, but $\alpha\notin L'_x$ because $g(\alpha)\neq g(x)$. We will not need this identification, only the existence of such a line $\ell_{x,\alpha}$, that varies non-trivially with $x$.

\begin{proof}[Proof of Proposition \ref{prop:inc_exceptional}]
    We only prove the ``only if'' direction, as we will not need the converse. The ``if'' direction may be proved by constructing explicit 1-parameter families in $\Inc^+(p,x)$, similarly to \cite[Proposition 3.9]{lian_pr}.
    
    Write $x=[1\cdot \gamma_1:1\cdot \gamma_2:\gamma_3]$ as above. If $\alpha=x$, then we are in case (i). Next, suppose that $\alpha\neq x$ but $g(\alpha)=g(x)$. Assume $\alpha_1\neq0$; otherwise, swap the roles of $\alpha_1$ and $\alpha_2$. Write
    \begin{equation*}
        u=[(\alpha_0\cdot 1_D)(\alpha_1v):(\alpha_0\cdot 1_D)(\alpha_2v):\alpha_3v1_D]\in \Inc^+(p,x)\subset Q^+_\beta(C,X).
    \end{equation*}
    Then,
    \begin{equation*}
        \gamma_3(\alpha_0\cdot 1_D)(\alpha_1v)-\gamma_1(\alpha_3v)=(\gamma_3\alpha_0\alpha_1-\gamma_1\alpha_3)v1_D
    \end{equation*}
    must vanish at $p$. However, if $g(\alpha)=g(x)$ but $\alpha\neq x$, and in addition $\alpha_1\neq0$, then we have $\gamma_3\alpha_0\alpha_1-\gamma_1\alpha_3\neq0$, so in fact $v1_D$ must vanish at $p$. Thus, we are either in case (ii), or $v$ must vanish at $p$ \emph{as a section of $\cL(-D)$}.
        
    Finally, assume that $v$ vanishes at $p$ as a section of $\cL(-D)$. Write
    \begin{equation*}
    (u,\phi)=[(\alpha_0\cdot 1_D+\epsilon v_0)(\alpha_1v+\epsilon v_1):(\alpha_0\cdot 1_D+\epsilon v_0)(\alpha_2v+\epsilon v_2):\alpha_3v1_D+\epsilon v_3],
    \end{equation*}
    as in \eqref{eqn: normal vector}, for the first-order expansion of a 1-parameter family of augmented quasimaps in $\Inc^+(p,x)$. We allow here the underlying ILS to vary to first order.
    
    Looking at the first order terms, we see that
    \begin{equation*}
        1_D(\gamma_2v_1-\gamma_1v_2),\gamma_3\alpha_0v_11_D-\gamma_1v_3
    \end{equation*}
    must both vanish at $p$. Recall from Proposition \ref{normal_bundle_canonical} that the triple $(1_Dv_1,1_Dv_2,v_3)$ defines a linear map
    \begin{equation*}
    \widetilde{\phi}:(\cO(-(\H-\E))^2\oplus\cO(-\H))|_\alpha\to W/v,
    \end{equation*}
    sending the basis vectors $e_1,e_2,e_3$ (defined before the proof of Proposition \ref{normal_bundle_canonical}) to the images of $1_Dv_1,1_Dv_2,v_3$ in $W/v$. Then, $\phi:\Omega_X(\log E)|_\alpha\to W/v$ is defined by restricting the domain. 
    
    Define now
    \begin{equation*}
         \ell_{x,\alpha}= \langle \gamma_2 e_1 - \gamma_1 e_2, \gamma_3 \alpha_0 e_1-\gamma_1 e_3 \rangle \cap \Omega_X(\log E)|_{\alpha} \subset (\cO(-(\H-\E))^2\oplus\cO(-\H))|_\alpha.
    \end{equation*}
    If $\alpha\neq x$, then the intersection $\ell_{x,\alpha}\subset \Omega_X(\log E)|_{\alpha}$ is a line, and it is clear by inspection that it depends on $x$ and only on $x$. By the calculation above, $\phi(\ell_{x,\alpha})\in W/v$ vanishes at $p$. 
    \end{proof}

\section{Transversality in dimension 2}\label{sec:dim2_proof}

In this section, we prove Theorem \ref{main_transversality}, which in particular implies Theorem \ref{thm:main_intro}. The main task is to show that incidence loci $\Inc(p_i,x_i)\subset \wtQ_d(C,\Bl_q(\bP^2))$ do not have unexpected intersections at complete quasimaps $u$ which have base-points at the $p_i$, and/or have non-generic bi-rank. To rule such intersections out, we twist down at base-points until $u$ defines (after possibly incorporating additional data from the underlying ILS) an honest morphism $u':C\to S$ of lower degree. However, $S$ is most naturally taken to be a (toric) \emph{blow-up} of $\Bl_q(\bP^2)$, so in order to control the geometry of $u'$, we will require a transversality result for maps to more general surfaces, \cite[Theorem 1.2]{cl_bn_surface}.

We recall this transversality result in \S\ref{sec:surface_BN}, and then specialize to the toric surfaces that we need in \S\ref{sec:toric_surface_BN}. We combine these ingredients to prove Theorem \ref{thm:main_intro} in \S\ref{sec:main_proof}.

\subsection{Curves on surfaces}\label{sec:surface_BN}

The main input in our proof of Theorem \ref{thm:main_intro} is the following.

\begin{theorem}\cite[Theorem 1.2]{cl_bn_surface}\label{thm:BN_surface}
Let $S$ be a smooth, projective surface. Let $\beta\in H_2(S)$ be a non-zero, effective curve class. Let $\cM_\beta(C,S)$ be the moduli space of maps $f:C
\to S$ in class $\beta$. Let $\cM^{\ge 4}_\beta(C,S)\subset \cM_\beta(C,S)$ be the open subset of maps $f$ whose scheme-theoretic image $C'$ has $[C']\cdot K_S^{\vee}\ge 4$. Then, $\cM^{\ge 4}_\beta(C,S)$ is generically smooth and pure of dimension
    \begin{equation*}
        \beta\cdot K^{\vee}_S+2(1-g),
    \end{equation*}
    the expected.
\end{theorem}

Theorem \ref{thm:BN_surface} is a consequence of standard results on Severi varieties, and may be viewed as a Brill-Noether theorem for curves on surfaces. Indeed, when $S=\P^2$, it recovers the classical Brill-Noether theorem, and when $S=\Bl_q(\P^2)=X$, it recovers Corollary \ref{cor:maps_expdim}, but with an additional generic smoothness statement. (We do not know if the generic smoothness holds for higher-dimensional blow-ups.)  

We will apply Theorem \ref{thm:BN_surface} to the cases of toric blow-ups of $\bP^2$, discussed in \S\ref{sec:toric_surface_BN}. See also \cite[Corollary 1.4]{cl_bn_surface} for a specialization of Theorem \ref{thm:BN_surface} to toric surfaces, though we will not need this statement here.

\subsection{Blow-ups of $\P^2$}\label{sec:toric_surface_BN}

Let $b_2:X_2\to\P^2$ be the blow-up of $\P^2$ at two distinct points. Let $\H,\E_1,\E_2\in H_2(X_2)$ be the classes of the line (pulled back from $\P^2$), and two exceptional divisors $E_1,E_2$, respectively. The multiplication table is shown below.

\begin{center}
\begin{tabular}{c | c c c }
 & $\H$ & $\E_1$ & $\E_2$ \\
\hline
$\H$ & $1$ & $0$ & $0$ \\ 
$\E_1$ & $0$ & $-1$ & $0$ \\  
$\E_2$ & $0$ & $0$ & $-1$    
\end{tabular}
\end{center}

\begin{lemma}\label{lem:x2_deg3}
    Let $C'\subset X_2$ be an integral curve. Then, $C' \cdot K_{X_2}^\vee\le 3$ if and only if $b_2(C')\subset\P^2$ is contained in a line.
\end{lemma}

\begin{proof}
Write $[C']=d \H - k_1 \E_1 - k_2 \E_2$, so that $C' \cdot K_{X_2}^\vee = 3d-k_1-k_2$. If $d=0$, then $(k_1,k_2)=(0,-1)$ or $(-1,0)$, and the right hand side is equal to 1. If $d=1$, then the right hand side is at most 3. Conversely, if $d >1$, because $C'$ is irreducible, we have $d \geq k_1+k_2$. It follows that $3d-k_1-k_2 \geq 2k_1+2k_2\ge 4$ as long as $k_1+k_2\ge 2$. The remaining cases can be checked by hand.
\end{proof}

Write now $\beta=d \H - k_1 \E_1 - k_2 \E_2\in H_2(X_2)$, and assume that $d>1$. By Lemma \ref{lem:x2_deg3}, the space $\cM^{\ge 4}_\beta(C,X_2)$ parametrizes maps whose post-composition with $b_2:X_2\to\P^2$ is given by linearly independent sections, spanning a linear series $W\in G^2_d(C)$. The two projections give two distinct subspaces $V_j\subset W$ vanishing along effective divisors $D_j$ of degree $k_j$. The following moduli space parametrizes this underlying data.

\begin{definition}\label{def:g_x2}
Let $\cW$ be the universal rank 3 bundle on $G^2_d(C)$. Define
\begin{equation*}
    G^{\circ}_{d,(k_1,k_2)}(C)\subset (\Gr(2,\cW)\times_{G^2_d(C)}\Gr(2,\cW))\times(\Sym^{k_1}(C)\times\Sym^{k_2}(C))
\end{equation*}
to be the \emph{locally} closed subscheme parametrizing the data of:
\begin{itemize}
    \item $W\in G^2_d(C)$, with underlying line bundle $\cL$,
    \item effective divisors $D_j\in \Sym^{k_j}(C)$, for $j=1,2$, and
    \item \emph{distinct} subspaces $V_1,V_2\subset W$ of dimension 2, such that $V_j$ vanishes along $D_j$.
\end{itemize}
We require furthermore that
\begin{itemize}
    \item $\Supp(D_1)\cap \Supp(D_2)=\emptyset$,
    \item For $j=1,2$, the maps $V_j\subset H^0(C,\cL(-D_j))\to H^0(C,\cL(-D_j)|_p)$ are surjective for every $p\in C$, and
    \item For $j=1,2$, the maps $W \to H^0(C,\cL|_p)$ are surjective for every $p\in\Supp(D_j)$.
\end{itemize}
\end{definition}

If $W,D_j,V_j$ come from $f\in \cM^{\ge 4}_\beta(C,X_2)$, then the last three conditions amount to base-point freeness. Note that the 1-dimensional subspace $V_1\cap V_2\subset W$ must vanish along $D_1+D_2$.

\begin{proposition}\label{prop:g_x2_expdim}
    $G^{\circ}_{d,(k_1,k_2)}(C)$ is generically smooth and pure of dimension 
    \begin{equation*}
        \beta\cdot K_{X_2}^{\vee}-2g-2=(3d-k_1-k_2)-2g-2,
    \end{equation*}
    the expected.
\end{proposition}

\begin{proof}
    By Lemma \ref{lem:x2_deg3}, we have a well-defined forgetful map $\pi_2:\cM^{\ge 4}_\beta(C,X_2)\to G^{\circ}_{d,(k_1,k_2)}(C)$ remembering $W,D_j,V_j$. In coordinates, we may write $f\in \cM^{\ge 4}_\beta(C,X_2)$ as
    \begin{equation*}
        f=[u_0^1u_0^2 u_1:u_0^1u_2:u_0^2u_3],
    \end{equation*}
    where $u_0^j$ cut out $f^{*}(E_j)\subset C$. Then, define $D_j=\text{div}(u_0^j)$ and
    \begin{align*}
        W&=\langle u_0^1u_0^2 u_1,u_0^1u_2,u_0^2u_3\rangle,\\
        V_1&=\langle u_0^1u_0^2 u_1,u_0^1u_2\rangle,V_2=\langle u_0^1u_0^2 u_1,u_0^2u_3\rangle.
    \end{align*}
    
    By construction, the map $\pi_2$ is surjective and smooth of relative dimension 4. The claim follows from Theorem \ref{thm:BN_surface}.
\end{proof}

We have a parallel story for the blow-up $X'_1$ of $X_1=\Bl_q(\P^2)$ at a point of the exceptional divisor of $X_1$. Let $b':X'_1\to\P^2$ be the composition of blow-ups. Let $\H\in H_2(X'_1)$ be the class of a line pulled back from $\P^2$. Let $\barE\in H_2(X'_1)$ be the class of the \emph{proper transform} $\overline{E}$ of the exceptional divisor on $X_1$, and let $\E\in H_2(X'_1)$ be the class of the exceptional divisor $E$ of the second blow-up. We have the following multiplication table.
\begin{center}
\begin{tabular}{c | c c c }
 & $\H$ & $\barE$ & $\E$ \\
\hline
$\H$ & $1$ & $0$ & $0$ \\ 
$\barE$ & $0$ & $-2$ & $1$ \\  
$\E$ & $0$ & $1$ & $-1$    
\end{tabular}
\end{center}
\begin{lemma}\label{lem:x'1_deg3}
    Let $C'\subset X'_1$ be an integral curve. Then, $C' \cdot K_{X'_1}^\vee\le 3$ if and only if $b'(C')\subset\P^2$ is contained in a line.
\end{lemma}

\begin{proof}
Write $[C']= d\H-(k+\bark)\barE-(2k+\bark) \E$, where the coefficients are chosen so that
\begin{equation*}
    [C']\cdot \H=d,[C']\cdot\barE=\bark,[C']\cdot\E=k.
\end{equation*}
We have $K^{\vee}_{X'_1}=3\H-(\E+\barE)-\E=3 \H-\barE-2\E$, so that $K^{\vee}_{S}\cdot\beta=3d-\bark-2k$. If $C'$ does not map to a line, then $b'(C')$ is a curve of degree $d\ge2$ intersecting the line through $q$ corresponding to the blown-up point of $X'_1$ with multiplicity $2k+\bark$ at $q$, so we have $d\ge 2k+\bark$. In this case, we have $[C'] \cdot K_{X_1'}^\vee =3d-2k-\bark \geq 2d \geq 4$. The converse is clear.
\end{proof}

We remark that the conclusions of Lemmas \ref{lem:x2_deg3} and \ref{lem:x'1_deg3} become false as soon as one passes to the blow-up $X_3$ of $\P^2$ at three points. For example, the proper transform of a conic $C'$ through the blown-up points of $\P^2$ has anti-canonical degree 3 \cite[Example 3.9]{cl_bn_surface}. However, $C'$ blows down to a line under the \emph{other} toric morphism $X_3\to\P^2$, contracting proper transforms of the three lines through pairs of blown-up points, that is, $C'$ is Cremona-equivalent to a line. A classification of integral curves of anti-canonical degree at most 3 on smooth, projective toric surfaces is given by \cite[Theorem 1.3]{cl_bn_surface}.

Returning to the surface $X'_1$, write $\beta=d\H-(k+\bark)\barE-(2k+\bark)\E\in H_2(X'_1)$, with $d>1$. As in the proof of Lemma \ref{lem:x'1_deg3}, we have
\begin{equation*}
    \beta\cdot \H=d,\beta\cdot\barE=\bark,\beta\cdot\E=k,
\end{equation*}
and $K^{\vee}_{X'_1}\cdot\beta=3d-\bark-2k$. A map $f\in \cM_\beta^{\ge 4}(C,X'_1)$ gives rise to the following data.

\begin{definition}\label{def:g'}
Let $\cW$ be the universal rank 3 bundle on $G^2_d(C)$.
Define
\begin{equation*}
    G^{'\circ}_{d,(k,\bark)}(C)\subset \Fl(\cW)\times(\Sym^{k}(C)\times\Sym^{\bark}(C))
\end{equation*}
to be the locally closed subscheme parametrizing the data of:
\begin{itemize}
    \item $W\in G^2_d(C)$,
    \item effective divisors $D,\barD$ of degrees $k,\bar{k}$,
    \item a flag $L\subset V\subset W$, such that $V\subset H^0(C,\cL(-D-\barD))$ and $L\subset H^0(C,\cL(-2D-\barD))$.
\end{itemize}
We require that
\begin{itemize}
    \item $V\to H^0(C,\cL(-D-\overline{D})|_p)$ is surjective for every $p\in C$,
    \item $W\to H^0(C,\cL|_p)$ is surjective for every $p\in\Supp(D)\cup\Supp(\overline{D})$, and
    \item $L\to H^0(\cL(-2D-\overline{D})|_p)$ is surjective for every $p\in\Supp(\overline{D})$.
\end{itemize}
\end{definition}

Indeed, write 
    \begin{equation*}
        f=[u_0^2\overline{u_0}u_1:u_0\overline{u_0}u_2:u_3] \in \cM_\beta^{\ge 4}(C,X'_1),
    \end{equation*}
where $u_0,\overline{u_0}$ cut out the pullbacks of $E,\overline{E}$ , respectively. Then, $D=\text{div}(u_0),\overline{D}=\text{div}(\overline{u}_0)$, and the flag $L\subset V\subset W$ given by 
    \begin{align*}
    \langle u_0^2\overline{u_0}u_1\rangle \subset \langle u_0^2\overline{u_0}u_1,u_0\overline{u_0}u_2\rangle \subset \langle u_0^2\overline{u_0}u_1,u_0\overline{u_0}u_2,u_3\rangle
    \end{align*}
define a point of $G^{'\circ}_{d,(k,\bark)}(C)$. The surjectivity hypotheses translate to the statement that the pairs $(u_0u_1,u_2),(u_0\overline{u_0},u_3),(u_1,\overline{u_0})$ of sections each share no common vanishing points. These ``base-point free'' conditions are in turn equivalent to the statement that the sections comprising $f$ define a morphism in class $\beta$ \cite{cox}.

\begin{proposition}\label{prop:g_x'1_expdim}
    The space $G^{'\circ}_{d,(k,\bark)}(C)$ is generically smooth and pure of dimension $(3d-\bark-2k)-2g-3$, the expected.
\end{proposition}

\begin{proof}
    The association above gives a map $\pi':\cM^{\ge 4}_\beta(C,X'_1)\to G^{'\circ}_{d,(k,\bark)}(C)$, which is surjective and smooth of relative dimension 5. The claim follows from Theorem \ref{thm:BN_surface}.
\end{proof}

\subsection{Conjecture \ref{conj:main} for $\Bl_q(\P^2)$}\label{sec:main_proof}
We return to the setting of complete quasimaps to $X=\Bl_q(\P^2)$. Fix non-negative integers $n_0,n_1,n'_1$ with $n_0+n_1+n'_1=n$, and an effective curve class $\beta=d\H-k\E\in H_2(X,\Z)$. Let $p_1,\ldots,p_n\in C$ be general points. Let $x_1,\ldots,x_{n_0}\in X$ be general points, let $L_{n_0+1},\ldots,L_{n_0+n_1}\subset X$ be general lines in class $\H$, and let $L'_{n_0+n_1+1},\ldots,L'_n\subset X$ be general lines in class $\H-\E$. 

\begin{theorem}\label{main_transversality}
The intersection
\begin{equation*}
    \left(\bigcap_{i=1}^{n_0}\Inc(p_i,x_i)\right) \cap \left(\bigcap_{i=n_0+1}^{n_0+n_1}\Inc(p_i,L_i)\right) \cap \left(\bigcap_{i=n_0+n_1+1}^{n}\Inc(p_i,L'_i)\right) \subset \wtQ_\beta(C,X)
\end{equation*}
is generically smooth and pure of codimension $2n_0+n_1+n'_1$, the expected. Furthermore, every general point of the intersection is contained in the open locus $\cM^\circ_\beta(C,X)\subset \wtQ_\beta(C,X)$.
\end{theorem}

That is, Conjecture \ref{conj:main} holds in dimension 2, and in particular, Theorem \ref{thm:main_intro} holds. 

The strategy of proof of Theorem \ref{main_transversality} is as follows. We construct an incidence correspondence varying over the space of all possible choices of $p_i,x_i,L_i,L'_i$, and show that relativized versions of the incidence loci $\Inc(p_i,x_i)$ intersect generically transversely. The same statement for general choices of incidence conditions follows. We obtain this transversality by restricting to bi-rank strata of $\wtQ_\beta(C,X)$ \eqref{Qtilde_strata}, and then to further strata defined by vanishing conditions of sections at the points $p_i$. Transversality along these finer strata is afforded by Propositions \ref{prop:g_x2_expdim} and \ref{prop:g_x'1_expdim}, proven in the previous section.

Let $(C^n)^\circ$ be the open subset of the $n$-fold product of $C$ with itself, where no two coordinates are equal. Let $X^\circ=X-E$ be the complement of the exceptional divisor. Let $\bP H^0(X,\H)^\circ$ be the open subset of the linear system of $\H$ consisting of irreducible lines, and let $\bP H^0(X,\H-\E)$ be the linear system of $\H-\E$. We allow the $p_i,x_i,L_i,L'_i$ to vary in these spaces.

Define the product
\begin{equation*}
    \Pi:=(C^n)^\circ \times (X^\circ)^{n_0} \times (\bP H^0(X,\H)^\circ)^{n_1} \times \bP H^0(X,\H-\E)^{n'_1}.
\end{equation*}
Form the global incidence loci $\cI(p_i,x_i),\cI(p_i,L_i),\cI(p_i,L'_i)\subset \wtQ_\beta(C,X)\times\Pi$ in the natural way relativizing $\Inc(p_i,x_i),\Inc(p_i,L_i),\Inc(p_i,L'_i)$, respectively. Define 
\begin{equation*}
    \cI=\left(\bigcap_{i=1}^{n_0}\cI(p_i,x_i)\right) \cap \left(\bigcap_{i=n_0+1}^{n_0+n_1}\cI(p_i,L_i)\right) \cap \left(\bigcap_{i=n_0+n_1+1}^{n}\cI(p_i,L'_i)\right)\subset \wtQ_\beta(C,X)\times\Pi.
\end{equation*}
We will show that $\cI$ is pure of expected codimension. In fact, we will prove that $\cI$ is pure of expected codimension upon restriction to various strata of $\wtQ_\beta(C,X)$, which we now define. 

Consider the poset of locally closed strata
\begin{equation}\label{Qtilde_strata}
\xymatrix{
 & Q^{(2,3)} \ar@{-}[dl] \ar@{-}[d] \ar@{-}[dr] & \\
Q^{(2,2)} \ar@{-}[d] \ar@{-}[dr] & Q^{(1,2),W} \ar@{-}[dl] \ar@{-}[dr] & E^{\circ} \ar@{-}[d] \ar@{-}[dl] \\
Q^{(1,2),V} & E^{V} & E^{W}
}
\end{equation}
where the poset relation is closure inclusion. The strata $Q^{(r_1,r_2)}$ are the strata on $Q^+_\beta(C,X)$, identified with their pre-images in $\wtQ_\beta(C,X)$, of augmented quasimaps $u$ of bi-rank $(r_1,r_2)$. The stratum $Q^{(1,2)}$ is further stratified by the locus $Q^{(1,2),V}$ where $u_3\in V$ and its complement $Q^{(1,2),W}$. The exceptional divisor $E_\beta(C,X)$, consisting of complete quasimaps $((\alpha,v),\phi)$, is stratified by the open locus $E^\circ$ where $\rank(\phi)=2$ and the closed loci $E^V$ (resp. $E^W$) where $\rank(\phi)=1$ and $\im(\phi)=V/v$ (resp. $\im(\phi)\neq V/v$). The open stratum $Q^{(2,3)}\subset \wtQ_\beta(C,X)$ is dense, and the strata appearing in the second and third rows of \eqref{Qtilde_strata} are pure of codimension 1 and 2, respectively.

\begin{proposition}\label{prop:stratum_inc_expdim}
    The restriction of $\cI$ to each stratum $Q\subset \wtQ_\beta(C,X)$ of \eqref{Qtilde_strata} is pure of codimension $2n_0+n_1+n'_1$.
\end{proposition}

Fix a stratum $Q\subset \wtQ_\beta(C,X)$, and let $\cJ\subset \cI\cap Q\subset \wtQ_\beta(C,X)$ be an irreducible component. Let $(u,\{p_i\},\{\Lambda_i\})$ be a general point of $\cJ$. We first make two reductions for the proof of Proposition \ref{prop:stratum_inc_expdim}.

\begin{reduction}\label{red:points}
    We may assume that $n_1=n'_1=0$, that is, all incidence conditions are imposed at points $x_i\in X$.    
\end{reduction}

\begin{proof}
    Assume that Proposition \ref{prop:stratum_inc_expdim} holds for $n_1=n'_1=0$. For simplicity of notation, we prove it holds for $(n_1,n'_1)=(1,0)$; the general case is similar (by induction). By assumption, every component of
    \begin{equation*}
        \cI_Q:=Q\cap \left(\bigcap_{i=1}^{n_0}\cI(p_i,x_i)\right)\subset Q\times (C^{n_0})^\circ\times(X^\circ)^{n_0}
    \end{equation*}
    is pure of codimension $2n_0$. The intersection 
    \begin{equation*}
        \cI^1_Q:=Q\cap \left(\bigcap_{i=1}^{n_0}\cI(p_i,x_i)\right)\cap \cI(p_{n_0+1},L_{n_0+1})\subset Q\times (C^{n_0+1})^\circ\times(X^\circ)^{n_0} \times \bP H^0(X,\H)^\circ
    \end{equation*}
    is obtained by pulling back $\cI_Q$, and intersecting with the pullback of $\cI(p_{n_0+1},L_{n_0+1})$ from $Q\times C\times \bP H^0(X,\H)^\circ$. 
    
    However, the restriction
    \begin{equation*}
        \cI(p_{n_0+1},L_{n_0+1})\subset \{(u,\{p_i\},\{x_i\}\} \times C\times \bP H^0(X,\H)^\circ
    \end{equation*}
    to \emph{any} point of $Q\times (C^{n_0})^\circ \times (X^\circ)^{n_0}$ is pure of codimension 1. Indeed, not every pair $(p_{n_0+1},L_{n_0+1})$ has the property that $u\in\Inc(p_{n_0+1},L_{n_0+1})$. Therefore, after removing further diagonals, it follows that $\cI^1_Q$ is pure of codimension $2n_0+1$.
\end{proof}

\begin{reduction}\label{red:bpf}
    Let $\pi(u)\in G_{d,k}(C)$ be the ILS underlying $u$. Then, we may assume that $\pi(u)$ is bpf, see \S\ref{sec:twisting}.
\end{reduction}

\begin{proof}
    We first apply Reduction \ref{red:points}. We will show that if Proposition \ref{prop:stratum_inc_expdim} holds for all smaller values of $\beta\cdot K_X^\vee=3d-k$, and if $\pi(u)$ fails to satisfy one of the properties (BPF1), (BPF2), then $\cJ$ cannot have larger than expected dimension. The reduction will therefore be valid by induction. The proof is similar to that of Proposition \ref{prop:inc_exp_dim}.
    
    Assume for the sake of contradiction that $\cJ$ has codimension strictly less than $2n$, and suppose first that $\pi(u)$ fails to satisfy (BPF1) at $p$. Then, we apply operation (T1) of Definition \ref{def:twisting}. The twisting makes sense at the level of the complete quasimap $u$, and may be performed in a neighborhood (in the analytic or \'{e}tale topologies) $\cJ^\circ$ of $u$. We obtain a map
    \begin{equation*}
        \epsilon:\cJ^\circ\to Q'\times \Pi
    \end{equation*}
    where the stratum $Q'\subset \wtQ_{d\H-(k+1)\E}(C,X)$ is the analogous one to $Q\subset \wtQ_{\beta}(C,X)$.

    First, assume that $p\neq p_1,\ldots,p_n$ generically on $\cJ$, so we may assume (shrinking if necessary) that this is the case everywhere on $\cJ^\circ$. The map $\epsilon$ has finite fibers, corresponding to choices of $p$ in the support of the twisted divisor $D$. Moreover, the image of $\epsilon$ lies in the intersection of the loci $\cI(p_i,x_i)$ in the image\footnote{Strictly speaking, if $Q$ is one of the strata supported in $E_\beta(C,X)$, then we use here the ``if'' direction of Proposition \ref{prop:inc_exceptional}, whose proof we omitted. One can sidestep this by instead viewing the loci $\cI(p_i,x_i)$ as \emph{defined} by the set-theoretic conditions given in Proposition \ref{prop:inc_exceptional}}. By the inductive hypothesis, we have that $\cI$ has codimension $2n$ in $Q'\times \Pi$. Thus, we have
    \begin{equation*}
        \dim(\im(\epsilon))\le \dim(Q'\times \Pi)-2n = \dim(Q\times \Pi)-2n-1< \dim(\cJ^\circ),
    \end{equation*}
    a contradiction.

    If instead we have $p=p_i$ for one (and necessarily only one) $i$, say $i=1$, then $\epsilon$ is injective, but we only have that the image of $\epsilon$ lies in the loci $\cI(p_i,x_i)$ for $i=2,\ldots,n$. By the inductive hypothesis, this imposes $2n-2$ conditions on $Q'\times(C^{n-1})^\circ\times X^{n-1}$. We have one additional, independent condition on the first factor of $C$, namely that $p_1$ lies in the support of the twisted divisor. We therefore have
    \begin{equation*}
        \dim(\im(\epsilon))\le \dim(Q'\times \Pi)-(2n-1) = \dim(Q\times \Pi)-2n< \dim(\cJ^\circ),
    \end{equation*}
    which is again a contradiction.
    
    If $\pi(u)$ fails to satisfy (BPF2) at $p$, then we apply the twisting operation (T2) instead of (T1). We obtain a map $\epsilon:\cJ^\circ\to Q'\times \Pi$, where $Q'\subset \wtQ_{(d-1)\H-(k-1)\E}(C,X)$. If $p\neq p_i$ for all $i$, then the fibers of $\epsilon$ have dimension 1, corresponding to choices of $p\in C$. On the other hand, the image of $\epsilon$ is contained in each $\cI(p_i,x_i)$, and we obtain a contradiction for dimension reasons.

    If instead $p=p_1$, then $\epsilon$ is injective. The image of $\epsilon$ lies in $\cI(p_i,x_i)$ for $i=2,\ldots,n$, and in addition lies in the codimension 1 locus $\cI(p_1,L'_1)$, where $L'_1\subset X$ is the unique line of class $\H-\E$ through $x_1$. A dimension count yields a contradiction once more.
\end{proof}

\begin{proof}[Proof of Proposition \ref{prop:stratum_inc_expdim}]
    Let $(u,\{p_i\},\{\Lambda_i\})\in \cJ\subset Q\times \Pi$ be a general point of an irreducible component. By Reduction \ref{red:points}, we may assume that the linear spaces $\Lambda_i$ are points $x_1,\ldots,x_n$, so $n_0=n$. By Reduction \ref{red:bpf}, we may restrict to the pullback of the bpf locus $G_{d,k}(C)^\circ$. We divide into three cases for $Q$.

    \underline{Case 1:} $Q\in \{Q^{(2,3)},Q^{(2,2)},E^\circ\}$. We claim that, for any $u\in Q$ with $\pi(u)$ bpf, the subvariety of $C\times X^\circ$ for which $u\in\Inc(p,x)$ has codimension 2. We conclude immediately that $\dim(\cJ)\le \dim(Q)+n=\dim(Q\times\Pi)-2n$.
    
    Indeed, if $Q\in \{Q^{(2,3)},Q^{(2,2)}\}$, then $[u_1(p_i):u_2(p_i)]\in \P^1$ must equal $g(x_i)$; this is one condition on $X$. If $u_0,u_3$ do not both vanish at $p_i$, then we get an additional condition at $X$, namely that $u(p_i)=x_i$. Otherwise, we get the additional condition on $C$ that $p_i$ be such a common vanishing point. If instead $Q=E^\circ$, then write $u=((\alpha,u_1),\phi)$. All three cases of Proposition \ref{prop:inc_exceptional} impose two conditions on $C\times X$; this is obvious in cases (i) and (ii). In case (iii), we have one condition on $p$ from $v$; we claim that $\phi$ imposes an additional condition on $x$. Indeed, because $\pi(u)$ is bpf, the kernel of the post-composition of $\phi$ with evaluation at $p_i$ must equal $\ell_{x,\alpha}$. This imposes a condition on $x$, because the line $\ell_{x,\alpha}$ varies with $x$.

    \underline{Case 2:} $Q\in \{Q^{(1,2),W},E^W\}$. To the complete quasimap $u$, we associate a non-zero section $v\in H^0(C,\cL(-D))$ and a subspace $V_2\subset W$ as follows:
    \begin{itemize}
    \item If $Q=Q^{(1,2),W}$, then write $u=[u_0u_1:u_0u_2:u_3]$, and take any $v\neq0$ in the 1-dimensional span of $u_1,u_2$. Let $V_2=\langle v,u_3\rangle$.
    \item If $Q=E^W$, then take $v$ to be that in the data $u=((\alpha,v),\phi)$. Let $V_2=\im(\phi)$, viewed as a 2-dimensional subspace of $W$ containing $v$. (Note that $\ker(\phi)=g^{*}\Omega_{\P^1}$.)
    \end{itemize}
    We have $V_2\neq V$ in both cases. Define $D_2$ to be the maximal divisor on which $v$ vanishes as a section of $\cL(-D)$ and every element of $V_2$ vanishes as a section of $\cL$. Write $k_2=\deg(D_2)$.

    Let $Q_{k_2}\subset Q$ be the locally closed subset where the construction above produces a divisor $D_2$ of degree exactly $k_2$. We claim that $Q_{k_2}\subset Q$ has codimension $k_2$, the expected. Indeed, for a point $u\in Q_{k_2}$, write $k_1,D_1,V_1$ for $k,D,V$. That $\pi(u)$ is bpf implies that $W,D_j,V_j$ give a point of $G^\circ_{d,(k_1,k_2)}(C)$. We therefore obtain a dominant map $\delta:Q_{k_2}\to G^\circ_{d,(k_1,k_2)}(C)$. Writing $t$ for the codimension of $Q\subset \wtQ_\beta(C,X)$, the relative dimension of $\delta$ is $4-t$. By Proposition \ref{prop:g_x2_expdim}, we have
    \begin{equation*}
        \dim(Q_{k_2})=\dim(G^\circ_{d,(k_1,k_2)}(C))+(4-t)=\dim(G^\circ_{d,k}(C))+(6-t)-k_2=\dim(Q)-k_2,
    \end{equation*}
    as claimed.

    After shrinking, we may view $\cJ$ as the closed subscheme of $Q_{k_2}\times (C^n)^\circ \times (X^\circ)^n$ cut out by the conditions $\cI(p_i,x_i)$. Suppose that $u\in Q_{k_2}$ underlies a point of $\cJ$. For any given $i$, there are two cases. Either $p_i\in D_2$ everywhere on $\cJ$, in which case $u$ imposes one condition on the corresponding factors $C\times X$ (namely that $p_i\in D_2$), or $p_i\notin D_2$ generically. In the latter case, $u$ imposes two conditions on $C\times X$, by a similar analysis to case 1. The first case occurs for at most $k_2$ indices among $i=1,\ldots,n$. It follows that
    \begin{equation*}
        \dim(\cJ) \le \dim(Q_{k_2}\times (C^n)^\circ \times (X^\circ)^n)-(2n-k_2) = \dim(Q\times (C^n)^\circ \times (X^\circ)^n)-2n,
    \end{equation*}
    as needed.

    \underline{Case 3:} $Q\in \{Q^{(1,2),V},E^V\}$. To the complete quasimap $u$, associate $v\in H^0(C,\cL(-D))$ as in Case 2. That is, if $Q=Q^{(1,2),V}$, then $v$ spans $\langle u_1,u_2\rangle$, and if $Q=E^V$, then $v$ is simply the element of $\P(\cV)$ underlying $u$. Let $D'\subset C$ be the (scheme-theoretic) intersection of the vanishing locus of $v\in H^0(C,\cL(-D))$ with $D$, let $k'$ be its degree, and define $\overline{D}=D-D'$. Define the flag $L=\langle v\rangle\subset V\subset W$.

    Let $Q_{k'}\subset Q$ be the locally closed subset where $D'$ has degree exactly $k'$. We claim that $Q_{k'}\subset Q$ has codimension $k'$, the expected. Indeed, the above construction gives rise to a dominant map $\delta:Q_{k'}\to G^{'\circ}_{d,(k',k-k')}(C)$ of relative dimension 3. (The divisor $D'$ becomes the divisor called $D$ in Definition \ref{def:g'}.) By Proposition \ref{prop:g_x'1_expdim}, we have
    \begin{equation*}
        \dim(Q_{k'})=\dim(G^{'\circ}_{d,(k',k-k')}(C))+3=\dim(G^\circ_{d,k}(C))+4-k'=\dim(Q)-k'.
    \end{equation*}

    Viewing, after shrinking, $\cJ\subset Q_{k'}\times (C^n)^\circ \times (X^\circ)^n$ as cut out by the conditions $\cI(p_i,x_i)$, a point $u\in Q_{k'}$ underlying a general point of $\cJ$ imposes one condition on $(p_i,x_i)\in C\times X$ if $p_i\in\overline{D}$ and two otherwise. It follows that
    \begin{equation*}
        \dim(\cJ) \le \dim(Q_{k'}\times (C^n)^\circ \times (X^\circ)^n)-(2n-k') = \dim(Q\times (C^n)^\circ \times (X^\circ)^n)-2n.
    \end{equation*}
\end{proof}

\begin{proof}[Proof of Theorem \ref{main_transversality}]
Every component of $\cI$ has codimension at most the expected in $\wtQ_\beta(C,X)$. By Proposition \ref{prop:stratum_inc_expdim}, every component $\cJ\subset \cI$ is contained generically in $Q^{(2,3)}$ and is of expected dimension. By Reduction \ref{red:bpf}, a general point $u\in \cJ$ has $\pi(u)\in G_{d,k}(C)^\circ$. It follows that $u\in\cM^\circ_\beta(C,X)$ as long as $u_0\neq0$. However, repeating the argument in the proof of Proposition \ref{prop:stratum_inc_expdim} shows that $\cI$ again has expected dimension upon restriction to the divisor on $Q^{(2,3)}$ where $u_0=0$.

We also have that $\cI$ is generically smooth. Indeed, by Corollary \ref{cor:maps_expdim}, $\cM^\circ_\beta(C,X)$ is generically smooth, and for any $u\in \cM^\circ_\beta(C,X)$, the locus of $(p_i,x_i)\in C\times X$ for which $u\in\Inc(p_i,x_i)$ is isomorphic to $C$. Theorem \ref{main_transversality} follows by restricting $\cI$ to general $p_i,\Lambda_i$.
\end{proof}

If $2n_0+n_1+n'_1=\dim(\wtQ_\beta(C,X))$, then Theorem \ref{main_transversality} shows that the integral
\begin{equation}\label{tev_integral_dim2}
    \int_{\wtQ_\beta(C,X)}[\Inc(p,x)]^{n_0}[\Inc(p,L)]^{n_1}[\Inc(p,L')]^{n'_1}
\end{equation}
counts the number of non-degenerate maps $f:C\to X$ in class $\beta$ with $f(p_i)=x_i$ for $i\in[1,n_0]$, $f(p_i)\in L_i$ for $i\in[n_0+1,n_0+n_1]$, and $f(p_i)\in L'_i$ for $i\in[n_0+n_1+1,n]$. 

These integrals are directly computable. It is immediate from their descriptions as degeneracy loci on $Q^+_\beta(C,X)$ that
\begin{align*}
[\Inc^+(p,L)]=h_{Q^+},&\quad[\Inc^+(p,L')]=h_{\cV^2}- \eta,\\
[\Inc^+(p,x)]&=h_{Q^+}(h_{\cV^2}- \eta).
\end{align*}
Here, $h_{Q^+},h_{\cV^2}$ denote the relative hyperplane classes on the projective bundles $Q^+,\P(\cV^2)$, respectively, and $\eta\in H^2(\Sym^k(C))$ is the class of locus $\{D|D-p\ge 0\}$. By Proposition \ref{prop:inc_lines}, the classes $[\Inc(p,L)],[\Inc(p,L')]$ are simply the pullbacks of the respective classes from $Q^+$. One can compute moreover that
\begin{equation*}
[\Inc(p,x)]=h_{Q^+}(h_{\cV^2}-\eta)- j_{*}(h_{\cV}-\eta),
\end{equation*}
where $j:E_\beta(C,X)\hookrightarrow \wtQ_\beta(C,X)$ is the inclusion, and $h_\cV$ is the pullback of the relative hyperplane class from $\P(\cV)$. Therefore, \eqref{tev_integral_dim2} is reduced to a tautological integral on $G^{2,0}_{d,k}(C)$, which is computable from the degeneracy locus description of $G^{2,0}_{d,k}(C)$.

In practice, the degeneration methods of \cite{lian_pr} may be more robust in computing $\Tev^{\Bl_q(\P^2)}_{g,n,\beta}$, even more so in higher dimension. We defer the computation to future work, though we mention the relationship to our earlier work on $\Tev^{\Bl_q(\P^r)}_{g,n,\beta}$ in a ``large degree'' regime in \S\ref{sec:other_spaces}.

\section{Remarks toward higher dimension}\label{sec:higher_dim}

In this section, we discuss, without proofs, how to extend the ideas of the previous two sections toward proving Conjecture \ref{conj:main} for all blow-ups $X=\Bl_{\P^s}(\P^r)$, highlighting in the end what is missing.

\subsection{Complete bi-collineations}

The objects underlying complete quasimaps to $\P^r$ \cite{lian_pr} are complete collineations \cite{vainsencher, thaddeus}. The objects underlying complete quasimaps to  $X=\Bl_{\P^s}(\P^r)$ are \emph{complete bi-collineations}.

Fix two inclusions $V\subset W$ and $U_1\subset U_2$ of vector spaces. We assume for simplicity that $\dim(U_1)=\dim(V)$ and $\dim(U_2)=\dim(W)$. The main example of interest is when  $V\subset W$ underlies an ILS in $G^{r,r-s-1}_{d,k}(C)$, and $U_1\subset U_2$ is the inclusion $H^0(X,\H-\E)\subset H^0(X,\H)$ induced by a choice of defining section $1_E$ (which is fixed).

\begin{definition}\label{def:bi-coll}
    A \emph{complete bi-collineation} $\phi:(U_1\subset U_2)\to (V\subset W)$ is a collection $\phi^0,\phi^1,\ldots,\phi^\ell$ of commutative squares of linear maps
    \begin{equation*}
        \xymatrix{
        U_1^j \ar[r]^{\phi^j_1} \ar[d] & V^j \ar[d]\\
        U_2^j \ar[r]^{\phi^j_2} & W^j 
        }
    \end{equation*}
    such that:
    \begin{itemize}
        \item for $j=0$, the vertical arrows are given by the given inclusions $U_1\subset U_2$ and $V\subset W$,
        \item for $j>0$, the vertical arrow $U_1^j\to U_2^j$ is given by the map $\ker(\phi^{j-1}_1)\to \ker(\phi^{j-1}_2)$ induced by $\phi^{j-1}$, and the vertical arrow $V^j\to W^j$  is given by the map $\rm{coker}(\phi^{j-1}_1)\to \rm{coker}(\phi^{j-1}_2)$ induced by $\phi^{j-1}$,
        \item for all $j$, the maps $\phi^j_1,\phi^j_2$ are not both zero, and
        \item for $j=\ell$, both maps $\phi^\ell_1,\phi^\ell_2$ are of full rank.
    \end{itemize}
Two complete bi-collineations $\phi,\phi'$ are considered equivalent if, for all $j$, the square $\phi'_j$ is obtained from $\phi_j$ by multiplying both horizontal maps $\phi^j_i$ by a scalar $t_j\in \C^*$.

The integer $\ell$ is referred to as the \emph{length} of the complete bi-collineation.
\end{definition}
Complete bi-collineations are a natural extension of complete collineations \cite{vainsencher,thaddeus}, which are sequences of single linear maps $\phi^j:U^j\to V^j$, with $U^j=\ker(\phi^{j-1}),V^j=\coker(\phi^{j-1})$. We will use the term ``bi-collineation'' without the modifier ``complete,'' for brevity.

The sequence $\phi^j$ of squares in a bi-collineation is guaranteed to terminate, exactly once a pair of full rank maps $\phi^\ell_1,\phi^\ell_2$ (not both zero) is reached. Indeed, the quantity 
\begin{equation*}
    \dim(U_1^j)+\dim(U_2^j)=\dim(V^j)+\dim(W^j)
\end{equation*}
is strictly decreasing in $j$. The vertical arrows, by definition, are determined at each step: the only interesting data are the horizontal arrows. We only label the arrows which are not already determined. The vertical maps $U_1^j\to U_2^j$ remain injective for all $j$, but the same is not necessarily true of $V^j\to W^j$.

We describe the points of $\wtQ_\beta(C,X)$ by describing the fibers of the iterated blow-up $b_Q:\wtQ_\beta(C,X)\to Q^+_\beta(C,X)$. Fix an augmented quasimap $u\in Q^+_\beta(C,X)$, and let $V\subset W$ be the underlying ILS. Write $U_1\subset U_2$ for the inclusion $H^0(X,\H-\E)\subset H^0(X,\H)$ induced by $1_E$. The sections underlying $u_j$ give rise to a commutative square
    \begin{equation}\label{Q+_square}
        \xymatrix{
        U_1 \ar[r]^{\phi^+_1} \ar[d] & V \ar[d]^{\lambda}\\
        U_2 \ar[r]^{\phi^+_2} & W
        }
    \end{equation}
where the vertical arrow $U_1\to U_2$ is the given inclusion, but the vertical arrow $\lambda:V\to W$ is a \emph{scalar multiple}, possibly zero, of the given inclusion. 

Indeed, $\phi_1^+$ sends $y_j\in U_1$ (\S\ref{sec:maps}) to $u_j\in V$ for $j=1,\ldots,r-s$. The map $\lambda\in\mathbb{C}$ is such that $u_0=\lambda\cdot 1_D$, and $\phi_2^+$ sends $1_Ey_j\in U_2$ to $u_0u_j\in W$ for $j=1,\ldots,r-s$ and $y_j$ to $u_j$ for $j=r-s+1,\ldots,r+1$. Accordingly, $u$ determines $\phi_1^+,\lambda,\phi_2^+$ only up to the $(\C^*)^2$-action given in Data \ref{data:map_to_X}. Fix a choice of square \eqref{Q+_square}.

Let $U_1^\perp=\langle y_{r-s+1},\ldots,y_{r+1}\rangle\subset U_2$ be the complement to $U_1\subset U_2$ spanned by torus-invariant sections cutting out divisors of class $\H$. Given a square \eqref{Q+_square}, define
\begin{equation*}
    \phi^0_2:U_2=U_1\oplus U_1^{\perp}\to W
\end{equation*}
by sum of the composition of $\phi^+_1:U_1\to V$ with the inclusion $V\subset W$, and the restriction of $\phi^+_2$ to $U_1^{\perp}$. Write also $\phi^0_1=\phi^+_1$. We obtain a commutative square
\begin{equation}\label{Q+_modified_square}
    \xymatrix{
    U_1 \ar[r]^{\phi_1^0} \ar[d] & V \ar[d] \\
    U_2 \ar[r]^{\phi_2^0} & W
    }
\end{equation}
where the vertical arrow $V\to W$ is now the given inclusion (and hence unlabelled). The map $\phi_2^0$ depends on the choice of square \eqref{Q+_square}, but the resulting description of complete quasimaps below is isomorphic given a different choice. 

\begin{ingredient}\label{prop:wtQ_points}
    Every blow-up center of $b_Q:\wtQ_\beta(C,X)\to Q^+_\beta(C,X)$ is smooth over $G^{r,r-s-1}_{d,k}(C)$. Let $u\in Q^+_\beta(C,X)$ be a point, to which we associate the squares \eqref{Q+_square} and \eqref{Q+_modified_square}. Then, a point in the fiber $b_Q^{-1}(u)$ is given by the data of a bi-collineation $\phi:(U_1\subset U_2)\to  (V\subset W)$, where $\phi_0$ is the square \eqref{Q+_modified_square}.
\end{ingredient}

The above is given as an ``Ingredient'' toward Conjecture \ref{conj:main} to indicate that it deserves a careful proof, but we do not provide one here.

The bi-collineation $\phi=\{\phi_j\}_{j=0}^{\ell}$ has length $\ell=0$ if and only if the underlying complete quasimap $u$ has bi-rank $(r-s,r+1)$. The data of $\phi^1$ can be identified with a normal vector to the stratum of augmented quasimaps of bi-rank $(\rank(\phi_1^0),\rank(\phi_2^0))$, by the same calculation as in Proposition \ref{normal_bundle_canonical}. The further data $\phi_2,\phi_3,\ldots$ comprising $\phi$ can be shown to arise from the iterated blow-ups, as in \cite[\S 2]{vainsencher}. More precisely, the bi-collineation $\phi$ arises from the blow-ups of proper transforms of the strata of bi-ranks
\begin{equation*}
    \left(\sum_{i=0}^{j-1}\rank(\phi^j_1),\sum_{i=0}^{j-1}\rank(\phi^j_2)\right),j=0,\ldots,\ell-1.
\end{equation*}
The data of $\phi^j$ is identified with a normal vector to the $j$-th blow-up center.

\begin{example}
    Take $X=\Bl_{q}(\P^2)$. For augmented quasimaps of bi-ranks $(2,2)$, $(1,2)$, and $(1,1)$, respectively, the squares $\phi^1$ take the following forms (from left to right):
    \begin{equation*}
        \xymatrix{
        0 \ar[r]^{\phi^1_1} \ar[d] & 0 \ar[d] \\
     U_1^\perp \ar[r]^{\phi^1_2} & W/V
        }
        \hspace{0.5in}
        \xymatrix{
        U_1^1 \ar[r]^{\phi^1_1} \ar[d]_{\sim} & V^1 \ar[d] \\
        U_2^1 \ar[r]^{\phi^1_2} & W^1
        }
        \hspace{0.5in}
                \xymatrix{
        g^{*}\Omega_{\P^1} \ar[r]^(0.55){\phi^1_1} \ar[d] & V/v \ar[d] \\
        \Omega_X(\log E)|_\alpha \ar[r]^(0.55){\phi^1_2} & W/v
        }
    \end{equation*}
    In the bi-rank $(2,2)$ case, $\phi^1_2$ must be a non-zero map between vector spaces of dimension 1, so $\phi^1$ adds no data; accordingly, the $(2,2)$-stratum is already a divisor on $Q^+_\beta(C,X)$. 
    
    In the bi-rank $(1,2)$ case, all four vector spaces that appear have dimension 1. $\phi^1_1$ determines $\phi^1_2$, so must be non-zero. On the locus $Q^{(1,2),W}$ where $\rm{im}(\phi_2^0)\neq V$, the induced map $V^1\to W^1$ is non-zero, so $\phi^1_2$ is also non-zero. On the locus $Q^{(1,2),V}$ where $\rm{im}(\phi_2^0)=V$, the induced map $V^1\to W^1$ is zero, so $\phi^1_2=0$. Thus, we have one additional square 
        \begin{equation*}
        \xymatrix{
        0 \ar[r]^{\phi^2_1} \ar[d] & 0 \ar[d] \\
      U_2^1=U_2^2 \ar[r]^{\phi^2_2} & W^2=W^1
        }
        \end{equation*}
        where $\phi^2_2\neq 0$. In both cases, the data of $\phi^j$ for $j>0$ is determined by $\phi^0$; indeed, the bi-rank $(1,2)$ stratum is again a divisor. $\phi$ has length 2 over $Q^{(1,2),V}$ because $Q^{(1,2),V}\subset \overline{Q^{(2,2)}}$.

    Finally, in the bi-rank $(1,1)$ case, we recover Proposition \ref{normal_bundle_canonical}. Indeed, at a point $(\alpha,v)\in X\times\P(\cV)$ of the bi-rank $(1,1)$ locus of $Q^+_\beta(C,X)$ (Proposition \ref{prop:Z_as_product}), $(U_1^1\subset U_2^{1})$ is identified by construction with $(g^{*}\Omega_{\P^1}\subset \Omega_X(\log E)|_\alpha )$, and $\phi^1_2$ determines $\phi^1_1$. If $\phi^1_2$ does not have full rank, then the further squares $\phi^j$ are determined up to equivalence.
\end{example}

\subsection{Incidence loci and Conjecture \ref{conj:main}}

We next describe the points of $\Inc(p,\Lambda)\subset\wtQ_\beta(C,X)$ in terms of the bi-collineations of Ingredient \ref{prop:wtQ_points}. We adopt the definitions of \S\ref{sec:incidence_loci}. Consider the \emph{dual} of the diagram \eqref{linear_spaces} defining $\Lambda$:
\begin{equation}\label{linear_spaces_dual}
\xymatrix{
   (\C^{(r-s-1)-\ell_1})^\vee \ar[d]_{\rho^\vee} \ar[r]^(0.4){\theta_1^{\vee}} &  H^0(X,\cO_X(\H-\E))=U_1 \ar[d]_{1_E} \\
   (\C^{r-\ell_2})^\vee \ar[r]^(0.4){\theta_2^\vee} & H^0(X,\cO_X(\H))=U_2
    }
\end{equation}
Write $S_1=(\C^{(r-s-1)-\ell_1})^\vee$ and $S_2=(\C^{r-\ell_2})^\vee$, so that $S_1=S_2\cap U_1\subset U_2$.

Working in the setting of Ingredient \ref{prop:wtQ_points}, fix an augmented quasimap $u\in \wtQ_\beta(C,X)$ and an associated square \eqref{Q+_square}. In particular, we identify $\lambda\in\mathbb{C}$ as the constant multiple of the given inclusion $V\subset W$ underlying the map $\lambda:V\to W$. Consider the map
\begin{equation*}
    (\lambda,1):U_1\oplus U_1^\perp \to U_1\oplus U_1^{\perp}=U_2
\end{equation*}
which multiplies by $\lambda$ in the first factor and is the identity in the second. 

Define
\begin{equation*}
    S_2^0=\langle(\lambda,1)^{-1}(S_2) , (S_1 \oplus 0)\rangle\subset U_1\oplus U_1^{\perp}.
\end{equation*}
Concretely, if $S_2\subset U_1\oplus U_1^{\perp}$ is spanned $(r-s-1)-\ell_1$ vectors in the coordinates underlying $U_1$ and $(s+1+\ell_1)-\ell_2$ additional vectors in all coordinates, then pulling back by $(\lambda,1)$ has the effect of multiplying the coefficients of coordinates underlying $U_1$ by $\lambda$. Adding $(S_1 \oplus 0)$ does nothing if $\lambda\neq0$, but if $\lambda=0$, then adding $(S_1\oplus 0)$ restores the vectors in the $U_1$-coordinates that vanish upon multiplying by $\lambda$. Write also $S_1^0=S_1\subset U_1$. By construction, we have $S_1^0=S_2^0\cap U_1$. 

Viewing $V\subset H^0(C,\cL(-D))$, let $V(-p)=H^0(C,\cL(-D-p))\cap V$ be the subspace of sections vanishing at $p$, which is either a hyperplane or all of $V$. Define similarly $W(-p)=H^0(C,\cL(-p))\cap W\subset H^0(C,\cL)$. 

Let $\phi\in b_Q^{-1}(u)$ be a bi-collineation as in Ingredient \ref{prop:wtQ_points}. For $j=0,\ldots,\ell$, and $i=1,2$, define:
\begin{itemize}
    \item $S_i^j=S_i^0\cap U_i^j$, for $i=1,2$, and 
    \item $V^j(-p)=\im(V(-p))\subset V^j$ and $W^j(-p)=\im(W(-p))\subset W^j$.
\end{itemize}
Consider the following two conditions on $\phi^j$ (cf. \cite[Definition 3.8]{lian_pr}):
\begin{enumerate}
    \item[$(\dagger)^j_1$] $\phi^j_1(S_1^j)\subset V^j(-p)$,
    \item[$(\dagger)^j_2$] $\phi^j_2(S_2^j)\subset W^j(-p)$.
\end{enumerate}

\begin{ingredient}\label{prop:inc_set_theoretic_general}
    We have $\phi\in \Inc(p,\Lambda)$ if and only if, for all $j=0,\ldots,\ell$, and for both $i=1,2$, at least one of the following is true:
    \begin{enumerate}
        \item $(\dagger)^i_j$ holds, or 
        \item the intersection $S_i^j=S_i\cap U_i^j$ is not transverse in $U_i$, that is, we have
        \begin{equation*}
    \dim(S_i\cap U_i^j)>\max(\dim(S_i)+\dim(U_i^j)-\dim(U_i),0)
\end{equation*} 
\end{enumerate}
(Different conditions between (1) and (2) may hold for different values of $i,j$.)
\end{ingredient}
When $j=0$, the intersection $S_i=S_i\cap U_i$ is trivially transverse for $i=1,2$, so $(\dagger)^0_1$ and $(\dagger)^0_2$ must both be satisfied. This is equivalent to $u\in \Inc^+(p,\Lambda)$. Tracing through Proposition \ref{prop:inc_exceptional} and its proof yields exactly the condition of Ingredient \ref{prop:inc_set_theoretic_general} in the case $X=\Bl_q(\P^2)$. In the case where the underlying augmented quasimap has bi-rank $(1,1)$ and $S_2^1=S_2^0\cap U_2^1$ has dimension 1, $S_2^1$ plays the role of the line $\ell_{x,\alpha}$ in Proposition \ref{prop:inc_exceptional}.

Ingredient \ref{prop:inc_set_theoretic_general} is entirely analogous to \cite[Proposition 3.9]{lian_pr}. In fact, adopting the notation of \cite{lian_pr}, Ingredient \ref{prop:inc_set_theoretic_general} asserts that $\phi\in \Inc(p,\Lambda)$ if and only if the complete collineation given by the maps $\phi^j_1$ (after removing redundant maps) lies in $\Inc(S_1,V(-p))\subset\Coll(U_1,V)$ and the complete collineation determined by the maps $\phi^j_2$ lies in $\Inc(S_2,W(-p))\subset\Coll(U_2,W)$. In order to verify if $\phi\in \Inc(p,\Lambda)$ in practice, one only has to check the condition $(\dagger)^i_j$ for at most one value of $j$ (for each $i=1,2$), see \cite[Lemma 3.10]{lian_pr}.

From here, one can hope to follow the same strategy of \cite[\S 4]{lian_pr} and \S\ref{sec:dim2_proof} of this work to prove Conjecture \ref{conj:main}. Namely, in the natural incidence correspondence, complete bi-collineations $\phi$ should impose the expected number of conditions on pairs $(p_i,\Lambda_i)$ if $\phi\in\Inc (p_i,\Lambda_i)$. However, this is only true if the images of the morphisms $\phi^j_i$ do not have ``unexpected vanishing'' along the points $p_i$. In the proof of Theorem \ref{main_transversality} in dimension 2, we have used (in cases 2 and 3) in a crucial way that the loci $Q_{k_2},Q_{k'}$ where the underlying ILS contain sections vanishing along many of the $p_i$ have expected codimension. 

For complete quasimaps to $\P^r$ (for $r$ arbitrary), the analogous statement is essentially afforded by the \emph{pointed} Brill-Noether theorem \cite[Lemma 4.4]{lian_pr}, in which Schubert conditions are imposed on linear series. However, an analog in the setting of ILS here is missing. We state the question imprecisely:
\begin{question}\label{question:pointed_BN_secant}
    Let $p_1,\ldots,p_n\in C$ be general points. Do loci on $G^{r,r-s-1}_{d,k}(C)$ of ILS containing flags of sections vanishing along sub-divisors of $p_1+\cdots+p_n$ have expected codimension?
\end{question}
In the same way that the needed transversality in dimension 2 came ultimately from Theorem \ref{thm:BN_surface} for the toric surfaces $X_2$ (Proposition \ref{prop:g_x2_expdim}) and $X'_1$ (Proposition \ref{prop:g_x'1_expdim}), it is plausible that an affirmative answer to Question \ref{question:pointed_BN_secant} could come from a Brill-Noether theorem for maps from curves to higher-dimensional toric varieties, and thereby provide the missing tool toward a proof of Conjecture \ref{conj:main}. We have raised the question of whether there could be such a Brill-Noether theorem in \cite[Question 1.5]{cl_bn_surface}.

\section{Relation to other moduli spaces}\label{sec:other_spaces}

In this section, we match, in a certain ``large $\beta$'' range, intersections on the moduli spaces of augmented quasimaps $Q^+=Q^+_\beta(C,\Bl_{\P^s}(\P^r))$ (and complete quasimaps $\wtQ_\beta(C,\Bl_{\P^s}(\P^r))$) to those on two other moduli spaces to which it maps. Set, as in \S\ref{sec:moduli_spaces}, $X=\Bl_{\P^s}(\P^r)$, and write $\beta=d\H^\vee+k\E^\vee$, where $d\ge k\ge 0$. For simplicity, we only deal with incidence conditions at points $x\in X$.

\subsection{Quasimaps}\label{sec:quasimaps_comparison}

Let $\barQ=\barQ_\beta(C,X)$ be the Ciocan-Fontanine--Kim space of quasimaps, \S\ref{sec:quasimaps}, and let $\psi:Q^+\to \barQ$ be the forgetful map defined in \S\ref{sec:augmented}.

Let $x\in X$ be a point not lying on the exceptional divisor, which we view as cut out by $r-s-1$ equations in $\alpha_1,\ldots,\alpha_{r-s}$ and $s+1$ additional equations in $\alpha_0\alpha_1,\ldots,\alpha_{r+1}$, as in \S\ref{sec:incidence_loci}. Let $p\in C$ be a point. Then, we may define $\overline{\Inc}(p,x)\subset \barQ$ by the vanishing of the corresponding linear combinations of the universal sections $\overline{u}_i$, in a parallel way to the definition of $\Inc^+(p,x)\subset Q^+$ by \eqref{Q+_degen1} and \eqref{Q+_degen2}. 

Namely, we have maps
\begin{equation}\label{Q-bar_degen1}
    \cO_{C\times \barQ}\to \nu_{*}(\barP_{d}\otimes\barP_{k}^{-1})^{r-s} \to \nu_{*}((\barP_{d}\otimes\barP_{k}^{-1})|_p)^{r-s-1}.
\end{equation}
and
\begin{equation}\label{Q-bar_degen2}
    \cO_{C\times \barQ}\to \nu_{*}(\barP_k\otimes(\barP_d\otimes\barP_k^{-1}))^{r-s}\oplus \nu_{*}(\barP_{d})^{s+1}\to \nu_{*}(\barP_d|_p)^{s+1}
\end{equation}
coming from the universal sections
\begin{align*}
    \overline{u}_1,\ldots,\overline{u}_{r-s}:&\cO_{C\times\barQ}\to \barP_d\otimes\barP_k^{-1},\\
    \overline{u}_0\overline{u}_1,\ldots,\overline{u}_{r+1}:&\cO_{C\times\barQ}\to \barP_d
\end{align*}
and the appropriate equations in these sections. We define the subscheme $\overline{\Inc}(p,x)\subset \barQ$ by the vanishing of \eqref{Q-bar_degen1} and \eqref{Q-bar_degen2}. By construction, the degeneracy loci of \eqref{Q-bar_degen1} and \eqref{Q-bar_degen2} pull back to \eqref{Q+_degen1} and \eqref{Q+_degen2}, respectively, so we have $\psi^{-1}\overline{\Inc}(p,x)=\Inc^+(p,x)$ as subschemes of $Q^+$.

In contrast to $Q^+$, the space of quasimaps $\barQ$ may have components of larger than expected dimension (necessarily not dominated by $Q^+$), see Example \ref{eg:bad_components}. More drastically, it may be the case that $\barQ$ contains components whose generic point has $\cL_k$ not effective, in which case $u_0=0$ everywhere. On the pathological components, the incidence loci $\overline{\Inc}(p,x)$ may also have smaller than expected codimension. The defining maps \eqref{Q-bar_degen1} and \eqref{Q-bar_degen2} nevertheless give rise to a natural virtual class
\begin{equation*}
    [\overline{\Inc}(p,x)]^{\vir}=c_1((\barP_{d}\otimes\barP_{k}^{-1})|_p)^{r-s-1}\cdot c_1((\barP_d|_p))^{s+1} \in H^{2r}(\barQ)
\end{equation*}
which pulls back to $[\Inc^+(p,x)]\in H^{2r}(Q^+)$. One can make sense of virtual integrals 
\begin{equation*}
    \int_{[\barQ]^{\vir}}([\overline{\Inc}(p,x)]^{\vir})^n,
\end{equation*}
against the virtual fundamental class of $\barQ$, but there is no obvious relationship in general to intersection numbers on $Q^+$.

\subsection{Naive quasimaps}\label{sec:naive_quasimaps}

We have previously given a partial computation of $\Tev^{\Bl_q(\P^r)}_{g,n,\beta}$ in \cite{cl2} on a certain moduli space $\P$, whose construction we recall. From the point of view of the present work, one could view $\P$ as a moduli space of ``naive quasimaps'', though this term was not used in \cite{cl2}. While the setting of the previous work was for blow-ups of projective spaces at points, the construction works equally well for blow-ups at linear spaces of any dimension.

We assume, as in \cite{cl2}, that $d-k>2g-1$, though the construction can be suitably adapted otherwise. In this case, we define $\P\to\Sym^k(C) \times \Pic^d(C)$ to be the projective bundle whose fiber over $(\cL,D)$ is $\P( H^0(\cL(-D))^{r-s}\oplus H^0(C,\cL)^{s+1})$. More precisely, let $\nu:C \times  \Sym^k(C) \times \Pic^d(C) \to \Sym^k(C) \times \Pic^d(C)$ be the projection, and as usual let $\cP_d,\cD$ denote a Poincar\'{e} line bundle (trivialized along a point) and a universal divisor. Then, we define $\P=\P(\cE)$, where
\begin{equation*}
\mathcal{E}=\nu_*(\mathcal{P}_d(-\cD))^{r-s} \oplus \nu_*(\mathcal{P}_d)^{s+1}
\end{equation*}
is a vector bundle on $\Sym^k(C)\times \Pic^d(C)$. Then, $\P$ is of expected dimension and a general point corresponds to a map $f:C\to X$ in class $\beta$.

We will denote the sections underlying a point of $\P$ by $f_1,\ldots,f_{r-s}\in H^0(\cL(-D))$ and $f_{r-s+1},\ldots,f_{r+1}\in H^0(C,\cL)$. There is a map $\delta:Q^+\to\P$ given by remembering $\cL,D$, and the tuple of sections 
\begin{equation*}
    [f_1:\cdots:f_{r+1}]=[u_0 u_1: \cdots, u_0u_{r-s}:u_{r-s+1}:\cdots:u_{r+1}].
\end{equation*}
More precisely, the push-forward of
\begin{equation*}
(u_0 u_1, \ldots, u_0u_{r-s},u_{r-s+1},\ldots:u_{r+1}): \cO_{C \times Q^+} \to \left(\cO_{Q^+}(1) \otimes (\mathcal{P}_d(-\mathcal{D}))^{r-s}\right) \oplus \left(\cO_{Q^+}(1) \otimes (\mathcal{P}_d)^{s+1}\right)
\end{equation*}
under $\nu:C \times Q^+ \to Q^+$ yields the injection
\begin{equation}\label{eqn: map for rho}
\cO_{Q^+}(-1) \hookrightarrow \nu_*(\mathcal{P}_d(-\mathcal{D}))^{r-s} \oplus \nu_*(\mathcal{P}_d)^{s+1}
\end{equation}
which defines $\delta$. Tautologically, we have $\delta^*(\cO_{\P}(1))=\cO_{Q^+}(1)$. Under $\delta$, the sections $u_0u_1,\ldots,u_0u_{r-s}$ coming from $Q^+$ are viewed as elements of $H^0(C,\cL(-D))$ after twisting back down by $1_D$, but this only recovers the $(r-s)$-tuple $(u_1,\ldots,u_{r-s})$ if $u_0\neq0$. 

Once more, let $p\in C$ and $x\in X$ be points, where we again view $x\notin E$ as cut out by linear equations in the coordinates $\alpha_j$ underlying $X$. Then, we define the subscheme $V(p,x)\subset \P$ by the vanishing of the corresponding $r-s-1$ sections in $H^0(C,\cL(-D))$ and $s+1$ sections in $H^0(C,\cL)$ at $p$. That is, $x\in X$ gives rise to a tautological map
\begin{equation*}
    \cO_{\P}(-1)\to \cE\to \nu_*(\mathcal{P}_d(-\cD)|_p)^{r-s-1} \oplus \nu_*(\mathcal{P}_d|_p)^{s+1}
\end{equation*}
whose degeneracy locus is $V(p,x)$, generically parametrizing the locus where $f(p)=x$. 

The associated cycle class is
\begin{equation*}
    [V(p,x)]=(c_1(\cO_{\P}(1))-\eta)^{r-s-1}\cdot c_1(\cO_{\P}(1))^{s+1}\in H^{2r}(\P),
\end{equation*}
where $\eta\in H^2(\Sym^k(C))$ is the class of the locus $\{D|D-p\ge0\}$.

The pullback $\delta^{-1}V(p,x)$ is \emph{not} equal to $\Inc^+(p,x)$. Indeed, whereas the incidence locus $\Inc^+(p,x)\subset Q^+$ imposes $r-s-1$ linear conditions on $u_1,\ldots,u_{r-s}$, the pullback $\delta^{-1}V(p,x)$ instead imposes the same equations on
\begin{equation*}
   \frac{u_0u_1}{1_D},\ldots,\frac{u_0u_{r-s}}{1_D},
\end{equation*}
which only recovers the same conditions if $u_0\neq0$. Therefore, $\delta^{-1}V(p,x)$ is given by the union of $\Inc^+(p,x)$ and an additional subscheme supported on the locus $\cD_0\subset Q^+$ where $u_0=0$. We see on the level of cycle classes that
\begin{equation*}
    \delta^{*}(c_1(\cO_{\P}(1))-\eta)=c_1(\cO_{Q^+}(1))-\eta=c_1(\cO_{\P(\cV^{r-s})}(1))+[\cD_0]-\eta,
\end{equation*}
which differs from the corresponding factor in 
\begin{equation*}
    [\Inc^+(p,x)]=(c_1(\cO_{\P(\cV^{r-s})}(1))-\eta)^{r-s-1}\cdot c_1(\cO_{Q^+}(1))^{s+1}\in H^{2r}(Q^+)
\end{equation*}
by the term $[\cD_0]$.

\subsection{The enumerative range}

We have the following diagram of moduli spaces:
\begin{equation*}
    \xymatrix{ 
    & Q^+ \ar[ld]_{\psi} \ar[rd]^{\delta} & \\
    \barQ & & \P
    }
\end{equation*}
We work in the setting of Definition \ref{def:tev} for $X=\Bl_{\P^s}(\P^r)$. The three quantities
\begin{equation}\label{3_integrals}
   \int_{[\barQ]^\vir}\prod_{i=1}^n[\overline{\Inc}(p_i,x_i)]^{\vir},\quad\int_{Q^+}\prod_{i=1}^n[\Inc^+(p_i,x_i)],\quad\int_{\P}\prod_{i=1}^n[V(p_i,x_i)].
\end{equation}
are intersection numbers of compactified loci of maps where $f(p_i)=x_i$, but in general have no immediate relationship. As explained in \S\ref{sec:quasimaps_comparison}, the first integrand pulls back to the second, but there is no reason to expect that $\psi_{*}[Q^+]=[\barQ]^\vir$, as $\psi$ may not be dominant. As explained in \S\ref{sec:naive_quasimaps}, the pullback $\delta^{-1}V(p_i,x_i)$ is in general strictly larger than $\Inc^+(p_i,x_i)$.

The pathologies go away if the loci where $u_0\neq0$ are avoided. Let $\barQ^{\star}\subset\barQ$ and $Q^{+\star}\subset Q^+$ be the open loci of (augmented) quasimaps where $u_0\neq 0$. Similarly, let $\P^{\star}\subset\P$ be the open locus where $f_1,\ldots,f_{r-s}$ are not all zero. Assume, as we have in the definition of $\P$, that $d-k>2g-1$. Assume further that $Q^+$ is non-empty, or equivalently that
\begin{align*}
\nonumber    d-g+1&\ge r+1,\\
    (d-k)-g+1&\ge r-s.
\end{align*}
 ($Q^+$ is always non-empty in cases of enumerative interest to us.) Then, $\barQ^\star,Q^{+\star},\P^{\star}$ are all smooth and irreducible of expected dimension, and the restrictions of the morphisms $\psi,\delta$ to $Q^{+\star}$ are birational, and surject onto $\barQ^{\star},\P^{\star}$, respectively. In fact, we have an isomorphism $\iota:\barQ^{\star}\to\P^{\star}$ defined by 
  \begin{equation*}
     \iota(u)=[u_0u_1:\cdots:u_0u_{r-s}:u_{r-s+1}:\cdots:u_{r+1}]
 \end{equation*}
 such that $\iota\circ \psi=\delta$. Upon restriction to these open subsets, $V(p,x)$ pulls back under $\iota,\delta$ to precisely $\overline{\Inc}(p,x),\Inc^+(p,x)$, respectively.

The basic observation is that, if $n>d$, then the intersections in question lie in the open loci $\barQ^\star,Q^{+\star},\P^{\star}$.

\begin{lemma}\label{intersections_match}
    Suppose that $n>d$ (which in particular implies $d-k>2g-1$) and that $Q^+$ is non-empty. Let $p_1,\ldots,p_n\in C$ and $x_1,\ldots,x_n\in X$ be general points, and assume \eqref{dim_constraint}. Then, we have
    \begin{equation*}
         \bigcap_{i=1}^n\overline{\Inc}(p_i,x_i) \subset \barQ^{\star},\quad \bigcap_{i=1}^n\Inc^+(p_i,x_i) \subset Q^{+\star},\quad\bigcap_{i=1}^n V(p_i,x_i)\subset \P^{\star}.
    \end{equation*}
    
    In particular, these intersections pull back to each other under $\psi,\delta,\iota$, and we have
    \begin{equation*}
        \int_{[\barQ]^\vir}\prod_{i=1}^n[\overline{\Inc}(p_i,x_i)]^{\vir}=\int_{Q^+}\prod_{i=1}^n[\Inc^+(p_i,x_i)]=\int_{\P}\prod_{i=1}^n[V(p_i,x_i)].
    \end{equation*}
\end{lemma}

\begin{proof}
A point of $\barQ$ or $Q^+$ with $u_0=0$ lying in the intersection in question would need to have $u_{r-s+1},\ldots,u_{r+1}$ vanish at all of $p_1,\ldots,p_n$. If $n>d$, then this requires all of $u_0,u_{r-s+1},\ldots,u_{r+1}$ to be identically zero, which is not allowed. Similarly, a point of $\P$ with $f_1=\cdots=f_{r-s}=0$ lying in the intersection in question would need to have $f_{r-s+1},\ldots,f_{r+1}$ vanish at all of $p_1,\ldots,p_n$, which is impossible for the same reason. This proves the first claim, as well as the compatibility of the intersections under pullback.

The equality of integrals on $Q^+,\P$ also follows immediately, because $\delta$ is a birational map of irreducible varieties, and the integrands are related by pullback up to a class supported on a subscheme of $Q^+$ disjoint from the intersection. The integrals on $\barQ$ and $\P$ also match, because the integrands are represented by intersections occurring entirely in the isomorphic open subsets $\barQ^{\star}$ and $\P^{\star}$, and given by cycle classes which are related by pullback. The virtual fundamental class $[\barQ]^\vir$ is needed in order for the integral on $\barQ$ to make sense globally, but is simply the usual fundamental class on $\barQ^{\star}$, where the intersection lies.
\end{proof}

In fact, when $n>d+1$, all three intersections enumerate $\Tev^X_{g,n,\beta}$, so Lemma \ref{intersections_match} merely asserts an equality of enumerative counts of curves. We continue to assume that non-degenerate maps $f:C\to X$ exist, so in particular, $d\ge r$ and $n>r+1$.

\begin{proposition}\label{tev_P}
    Suppose that $n>d+1$. Then, the intersection $\bigcap_{i=1}^n V(p_i,x_i)\subset\P$ is transverse and consists precisely of the points enumerated in $\Tev^X_{g,n,\beta}$. In particular, the same is true of all three integrals appearing in Lemma \ref{intersections_match}.
\end{proposition}
\begin{proof}
    Let $f \in \bigcap_{i=1}^n V(p_i,x_i)$ be a point. We perform the twisting procedure of \cite[\S 3.2.1]{cl2} to the sections underlying $f$ to produce a map $\overline{f}:C\to X$ in class $\beta'$, in a simplified form. Consider the operations, parallel to those in Definition \ref{def:twisting}:
    \begin{enumerate}
        \item[(T1)] If $f_1,\ldots,f_{r-s}\in H^0(C,\cL(-D))$ all vanish at $p$, then replace $D$ with $D+p$ and twist $f_1,\ldots,f_{r-s}$ down.
        \item[(T2)] If $p\in\Supp(D)$ and $f_{r-s+1},\ldots,f_{r+1}\in H^0(C,\cL)$ all vanish at $p$, then replace $\cL$ with $\cL(-p)$, replace $D$ with $D-p$ and twist all sections accordingly.
    \end{enumerate}

    Suppose that $f$ does not already define a map $C\to X$ in class $\beta$, that is, $f$ is not already bpf. Then, the operations (T1) and (T2) may be applied, in that order, repeatedly until it does. As above, the hypothesis that $n>d+1$ guarantees that not all of $f_1,\ldots,f_{r-s}$ are identically zero. As in the proof of Proposition \ref{prop:inc_exp_dim}, we track the number of linear conditions imposed on $f$, as well as the dimension of the family in which $f$ moves, at every step. Unless $f$ was already bpf, the \emph{expected} number of moduli in the end will be strictly less than the number of linear conditions imposed on $f$, cf. \cite[Lemma 49]{cl2}. We claim that, after the twisting is completed, the resulting map $\overline{f}:C\to X$ is non-degenerate, so by Corollary \ref{cor:maps_expdim}, the expected number of moduli for $f$ is the actual number. The usual incidence correspondence argument shows that $f$ was bpf to begin with, and the transversality claim for $\P$ also follows.

    We now prove the claim that $\overline{f}$ is non-degenerate. First, observe that if (T1) is applied at one of the points $p_i$, for $i=1,\ldots,n$, then (T2) must be applied at $p_i$ afterward. Indeed, in order for $f\in V(p_i,x_i)$, if $f_1,\ldots,f_{r-s}$ vanish at $p_i$, then so do $f_{r-s+1},\ldots,f_{r+1}$.

    Next, let $\ell$ be the number of $p_i$ among $p_1,\ldots,p_n$ at which (T2) is applied. We claim that $\ell\le n-r-1$. Note that (T2) decreases the value of $d$ by 1, and neither operation increases $d$, so $b\circ \overline{f}:C\to\P^r$ has degree at most $d-\ell$. On the other hand, of the remaining $n-\ell$ points $p_i$, \emph{neither} (T1) nor (T2) could have been applied, so it must be the case that $\overline{f}(p_i)=x_i$ for these $p_i$. The image $b(\overline{f}(C))\subset\P^r$ must therefore be a curve of degree at most $d-\ell \le n-2-\ell$ passing through $n-\ell$ general points. If $\ell\ge n-r$, then we obtain a contradiction, because every curve of degree at most $n-2-\ell\le r-2$ in $\P^r$ lies in a linear space of dimension $n-2-\ell$, which cannot contain $n-\ell$ general points.

    Therefore, for at least $r+1$ of the points $p_i$, it is the case that $\overline{f}(p_i)=x_i$. Because the $x_i$ are assumed general, it follows that $\overline{f}$ is non-degenerate, because $r+1$ general points in $\P^r$ lie in no common hyperplane. As explained above, Corollary \ref{cor:maps_expdim} combined with the usual incidence correspondence argument shows that the intersection number on $\P$ is enumerative.

    We immediately deduce the same transversality for the intersection $\bigcap_{i=1}^n\overline{\Inc}(p_i,x_i)\subset\barQ$ by pulling back by $\iota$. Furthermore, if $n>d+1$, then we see from taking $\ell=0$ in the argument above that $\bigcap_{i=1}^n V(p_i,x_i)$ consists of \emph{non-degenerate} maps, and therefore pulls back under $\delta$ to a transverse intersection on $Q^+$ enumerating non-degenerate maps in class $\beta$ as well.
\end{proof}

Because the intersection $\bigcap_{i=1}^n\Inc^+(p_i,x_i)\subset Q^+$ takes place in the non-degenerate locus, we also obtain the same transversality of incidence loci on the space of \emph{complete} quasimaps to $X$ in the range $n>d+1$. The inequality $n>d+1$ is equivalent to $d>(r-s-1)k+rg$, so this is a transversality statement for ``large'' curve classes $\beta$, in line with the asymptotic enumerativity phenomena for stable maps studied in \cite{lian_pand,bllrst}.

When $s=0$, the integrals $\int_{\P}\prod_{i=1}^n[V(p,x)]$ have been computed explicitly in \cite[\S 3.4]{cl2}\footnote{The computation should be no different when $s>0$.}, and shown to enumerate $\Tev^{X}_{g,n,\beta}$ when $n\ge d+g+1$. The improvement to the bound $n>d+1$ comes from the application of Corollary \ref{cor:maps_expdim} in place of the weaker \cite[Proposition 13]{lian_pand}. The fact that, when $n>d+1$, the intersection of incidence loci $\overline{\Inc}(p_i,x_i)$ on $\barQ$ is transverse and occurs away from the boundary (and in particular computes $\Tev^{X}_{g,n,\beta}$), may be viewed as an asymptotic enumerativity statement for virtual invariants on $\barQ$.

There seems to be no reason that the integrals appearing in Lemma \ref{intersections_match} should agree outside the range $n>d$, when excess intersections begin to disagree on the level of schemes. The only space on which one can reasonably expect tautological intersections to compute enumerative curve counts  $\Tev^{X}_{g,n,\beta}$ for \emph{all} $n,\beta$ is the space of complete quasimaps.

 \bibliographystyle{alpha} 
 \bibliography{bib}

\end{document}